\documentclass[final]{siamltex}
\usepackage{times}
\usepackage{algorithmicx, algpseudocode}
\usepackage{algorithm}
\usepackage{color}
\usepackage{graphicx}
\usepackage[font={small}]{caption}
\usepackage{cleveref}
\usepackage{ulem}
\usepackage{verbatim}
\usepackage{wrapfig,lipsum,booktabs}
\usepackage[mathscr]{euscript}
\usepackage{amssymb}%''''
\usepackage{amsmath}%facilitate writing math formulas and improves the typographical quality of their output.
\usepackage{multirow}
\usepackage{bbm}
\usepackage[titletoc,title]{appendix}
\usepackage{mathabx}
\usepackage{spreadtab}
\usepackage{diagbox}

\usepackage{mathtools}
\DeclarePairedDelimiter{\ceil}{\lceil}{\rceil}

\def\nh{{\small \sc \bf Nh2D}}%nh2D100100
\def\sky{{\small \sc \bf Sky2D}}%sky2D100100
\def\skyt{{\small \sc \bf Sky3D}}%sky3D202020
\def\ani{{\small \sc \bf Ani3D}}%ani3d202020
\def\poiss{{\small \sc \bf Poisson2D}}%poisson(100,100)
\def\nho{{\small \sc  Nh2D}}%nh2D100100
\def\skyo{{\small \sc  Sky2D}}%sky2D100100
\def\skyto{{\small \sc  Sky3D}}%sky3D202020
\def\anio{{\small \sc  Ani3D}}%ani3d202020
\def\poisso{{\small \sc  Poisson2D}}%poisson(100,100)
\newcommand{\sq}{\hbox{\rlap{$\sqcap$}$\sqcup$}}
\newcommand{\qed}{\hspace*{\fill}\sq}

\newlength{\algrhswidth}
\title{Numerical Stability of s-step Enlarged Krylov Subspace Conjugate Gradient methods}

\author{Sophie Moufawad \thanks{American University of Beirut (AUB), Beirut, Lebanon.  (sm101@aub.edu.lb) } }

\begin{document}

\maketitle \thispagestyle{empty}

%=========================================================================
%  Abstract
%=========================================================================
\begin{abstract}
Recently, enlarged Krylov subspace methods, that consists of enlarging the Krylov subspace by a maximum of $t$ vectors per iteration based on the domain decomposition of the graph of $A$, were introduced in the aim of reducing communication when solving systems of linear equations $Ax=b$. In this paper, the s-step enlarged Krylov subspace Conjugate Gradient methods are introduced, whereby $s$ iterations of the enlarged Conjugate Gradient methods are merged in one iteration. The numerical stability of these s-step methods is studied, and several numerically stable versions are proposed.
  Similarly to the enlarged Krylov subspace methods, the s-step  enlarged Krylov subspace methods have a faster convergence than Krylov methods, in terms of iterations. Moreover, by computing $st$ basis vectors of the enlarged Krylov subspace $\mathscr{K}_{k,t}(A,r_0)$ at the beginning of each s-step iteration, communication is further reduced. It is shown in this paper that the introduced methods are parallelizable with less communication, with respect to their corresponding enlarged versions and to Conjugate Gradient.   
\end{abstract}


%\begin{AMS}\end{AMS}

\section{Introduction}
\addtolength{\belowdisplayskip}{-3mm}
\addtolength{\abovedisplayskip}{-3mm}

Recently, the enlarged Krylov subspace methods \cite{EKS} were introduced in the aim of obtaining  methods that converge faster than classical Krylov methods, and are parallelizable with less communication, whereby communication is the data movement between different levels of memory hierarchy (sequential) and different processors (parallel). 
Different methods and techniques have been previously introduced for reducing communication in Krylov subspace methods such as Conjugate Gradient (CG) \cite{cgor}, Generalized Minimal Residual (GMRES) \cite{ssch}, bi-Conjugate Gradient \cite{bicg1,bicg2}, and bi-Conjugate Gradient Stabilized \cite{bicgstab}. The interest in such methods is due to the communication bottleneck on modern-day computers and the fact that the Krylov subspace methods are governed by BLAS 1 and BLAS 2 operations that are communication-bound. 

These methods and techniques can be categorized depending on how the communication reduction is achieved. There are three main categories where the reduction is achieved at the mathematical/theoretical level, algorithmic level, and implementation level. 
The first category is introducing  methods based on different Krylov subspaces %mathematical concepts 
such as the augmented Krylov methods \cite{augmKSM, augmKSM2}, and the Block Krylov methods \cite{bcg} that are based on the augmented and block Krylov subspaces respectively. The recently introduced enlarged Krylov subspace methods fall into this category since the methods search for the approximate solution in the enlarged Krylov subspace. The second category is to restructure the algorithms such as the s-step methods that compute $s$ basis vectors per iteration \cite{sstepcg1, chronop,walker, erhel, Carson} and the communication avoiding methods that further reduce the communication \cite{cagmres, hoemmen, grigori}. The third category is to hide the cost of communication by overlapping it with other computation, like pipelined CG \cite{hidecg,hidecg2} and  pipelined GMRES \cite{hide}.

In this paper, we introduce the s-step enlarged Krylov subspace methods, whereby $s$ iterations of the enlarged Krylov subspace methods are merged in one iteration.  The idea of s-step methods is not new, as mentioned previously. However, the aim of this work is the introduction of  methods that reduce communication with respect to the classical Krylov methods, at the three aforementioned levels (mathematical/theoretical, algorithmic, and implementation level).
 Similarly to the enlarged Krylov subspace methods , the s-step  enlarged Krylov subspace methods have a faster convergence than Krylov methods, in terms of iterations. In addition, computing $st$ basis vectors of the enlarged Krylov subspace $\mathscr{K}_{k,t}(A,r_0)$ at the beginning of each s-step enlarged Krylov subspace iteration reduces the number of sent messages in a distributed memory architecture.

We introduce several s-step enlarged Conjugate Gradient versions, based on the short recurrence enlarged CG methods (SRE-CG and SRE-CG2) and MSDO-CG presented in \cite{EKS}. After briefly introducing CG, a Krylov projection method for symmetric (Hermitian) positive definite (SPD) matrices, and the enlarged CG methods in section \ref{sec:overview}, we discuss the new s-step enlarged CG versions (section \ref{sec:sstep}) in terms of numerical stability (section \ref{sec:SRNum}), preconditioning (section \ref{sec:precCG}), and communication reduction in parallel (section \ref{sec:par}). Although we only consider in this article a distributed memory system, however the introduced methods reduce communication even in shared memory systems. Finally we conclude in section \ref{sec:conc}.

%=====================================================================
\section{From Conjugate Gradient (CG) to Enlarged Conjugate Gradient Methods}\label{sec:overview}
%=====================================================================
The Conjugate Gradient method of Heistens and Stiefel  \cite{cgor} was introduced in 1952. Since then, different  CG versions were introduced for different purposes. In 1980, Dian O'Leary introduced the block CG method \cite{bcg}  for solving an SPD system with multiple right-hand sides. Block CG performs less work than solving each system apart using CG. In addition, it may converge faster in terms of iterations and time in some cases discussed in \cite{bcg}. In 1989, Chronopoulos and Gear introduced the s-step CG method, that performs $s$ CG iterations simultaneously with the goal of reducing communication by performing more flops using the data in fast memory. Several CG versions where introduced for solving successive linear systems with different right-hand sides, by recycling the computed Krylov subspace, such as \cite{erhelASCG}. Moreover, several preconditioned and parallelizable CG versions were introduced, such as deflated CA-CG \cite{Carson2}, MSD-CG \cite{msdcg}, augmented CG \cite{augmKSM, augmKSM2}. Recently, enlarged Conjugate Gradient methods such as SRE-CG, SRE-CG2, and MSDO-CG were introduced \cite{EKS}. In this section, we briefly discuss CG, s-step CG, and enlarged CG versions. For a brief overview of other related CG versions such as the block CG, coop-CG, and MSD-CG, refer to \cite{sophiethesis}.  

The CG method is a Krylov projection method that finds a sequence of approximate solutions $x_k \in x_0 + \mathcal{K}_k(A,r_0)$ ($k>0$) of the system $Ax = b$, by imposing the Petrov-Galerkin condition, $\; r_k \perp \mathcal{K}_k$,  where $\mathcal{K}_k (A, r_0) = span\{ r_0, Ar_0, A^2 r_0,..., A^{k-1} r_0 \}$ is the Krylov subspace of dimension $k$, $x_0$ is the initial iterate, and $r_0$ is the initial residual. At the $k^{th}$ iteration, CG computes  the new approximate solution  $x_k = x_{k-1} +\alpha_k p_k$ that minimizes $\phi (x) = \frac12 (x)^t A x - b^t x$ over the corresponding space $x_0 + \mathcal{K}_k(A, r_0)$, where $k>0$, $p_k= r_{k-1}+\beta_k p_{k-1} \in \mathcal{K}_k(A, r_0)$ is the $k^{th}$ search direction,  $p_1 = r_0$, and $\alpha_{k} = \frac{(p_{k})^t r_{k-1}}{(p_{k})^t A p_{k}} = \frac{||r_{k-1}||^2_2} {||p_{k}||_A^2}$ is the step along the search direction. As for $\beta_k  = - \frac{(r_{k-1})^t A p_{k-1}} {(p_{k-1})^t A p_{k-1}} = \frac{||r_{k-1}||^2_2}{||r_{k-2}||^2_2}$, it is defined so that the search directions are A-orthogonal ($p_k^tAp_i = 0$ for all $i\neq k$), since otherwise the  Petrov-Galerkin condition is not guaranteed. 

The s-step CG method \cite{chronop} introduced by Chronopoulos and Gear in 1989 is also a Krylov projection method that solves the system $Ax = b$ by imposing the Petrov-Galerkin condition. However, it finds a sequence of approximate solutions $x_{k} \in x_0 + \mathcal{K}_{sk}(A,r_0)$, where  $k>0$, $s>0$, and $\mathcal{K}_{sk}(A,r_0) = span\{ r_0, Ar_0, A^2 r_0,..., A^{sk-2} r_0, A^{sk-1} r_0 \}$.  At the $k^{th}$ iteration,  $x_{k} = x_{k-1} - P_k\alpha_k$ is obtained by minimizing $\phi (x)$, where $P_k$ is a matrix containing the $s$ search directions and $\alpha_k = ((P_{k})^t A P_{k})^{-1} P_k^tr_{k-1} $ is a vector containing the $s$ corresponding step lengths. Initially, $P_1 = R_0 = [r_0 \; Ar_0 \;... \; A^{s-1}r_0]$ is defined as the first $s$ basis vectors of the Krylov subspace. Then $P_k = R_{k-1}+P_{k-1}\beta_k$ for $k>1$, where $R_{k-1} = [r_{k-1} \;Ar_{k-1}\; ... \;A^{s-1}r_{k-1}]$, and $\beta_k = - (P_{k-1}^tAP_{k-1})^{-1}(R_{k-1}^tP_{k-1})$ is an $s \times s$ matrix. %computed at iteration $i$ of the s-step CG method.   

On the other hand, the enlarged CG methods are enlarged Krylov projection methods that find a sequence of approximate solutions $x_k \in x_0 + \mathscr{K}_{k,t}(A,r_0)$ ($k>0$) of the system $Ax = b$, by imposing the Petrov-Galerkin condition, $\; r_k \perp \mathscr{K}_{k,t}(A,r_0)$, where $\mathscr{K}_{k,t} (A, r_0) = span\{ T(r_0), AT(r_0), A^2 T(r_0),..., A^{k-1} T(r_0) \}$ is the enlarged Krylov subspace of dimension at most $tk$, $x_0$ is the initial iterate, $r_0$ is the initial residual, and $T(r_0)$ is an operator that splits $r_0$ into $t$ vectors based on a domain decomposition of the matrix $A$. Several enlarged CG versions were introduced in \cite{EKS}, such as MSDO-CG, LRE-CG, SRE-CG, SRE-CG2, and the truncated SRE-CG2.
Moreover, in \cite{ECG} a block variant of SRE-CG is proposed, whereby the number of search directions per iteration is reduced using deflation. 

%=====================================================================
\section{s-step Enlarged CG versions}\label{sec:sstep}
%=====================================================================
 The aim of s-step enlarged CG methods is to merge $s$ iterations of the enlarged CG methods, and perform more flops per communication, in order to reduce communication.
 
 In the case of the SRE-CG and SRE-CG2 versions, reformulating into s-step versions is straight forward since these methods build an A-orthonormal basis $\{T(r_0), AT(r_0),...,A^kT(r_0)\}$ and update the approximate solutions $x_k$. The basis construction is independent from the consecutive approximate solutions. But the challenge is in constructing a numerically stable A-orthonormal basis of $st$ vectors, where $t$ is number of domains and $s$ is the number of merged iterations.
 
 As for MSDO-CG, at each iteration $k$, $t$ search directions are built and A-orthonormalized and used to update the approximate solution. Moreover, the construction of the search directions depends on the previously computed approximate solution. So merging $s$ iterations of the MSDO-CG algorithm requires more work, since it is not possible to separate the search directions construction from the solution's update. Hence, a new version will be proposed where a modified enlarged Krylov subspace is built.
 
 \subsection{s-step SRE-CG and SRE-CG2} \label{sec:sstepS}
 The short recurrence enlarged CG (SRE-CG) and SRE-CG2 methods, are iterative enlarged Krylov subspace projection methods that build at the  $k^{th}$ iteration, an A-orthonormal ``basis" $Q_k$ ($Q_k^tAQ_k = I$) for the enlarged Krylov subspace \vspace{-5mm}$$\mathscr{K}_{k,t} = span\{T(r_0), AT(r_0),...,A^{k-1}T(r_0)\},$$ and approximate the solution, $x_k = x_{k-1}+Q_k\alpha_k$, by imposing the orthogonality condition on $r_k = r_{k-1} - AQ_k\alpha_k$, ($r_k \perp \mathscr{K}_{k,t}$), and minimizing $$\phi(x) = \frac{1}{2}x^tAx - x^tb,$$ where $Q_k$ is an $n\times kt$ matrix and $T(r_0)$ is the set of $t$ vectors obtained by projecting $r_0$ on the $t$ distinct domains of $A$. 
 
There are 2 phases in these methods, building the ``basis" and updating the approximate solution. The difference between SRE-CG and SRE-CG2 is in the ``basis" construction. After A-orthonormalizing $W_1 = \mathscr{T}_0$, where $\mathscr{T}_0$ is the matrix containg the $t$ vectors of $T(r_0)$, it is shown in \cite{EKS} that at each iteration $k\geq 3$, $W_k = AW_{k-1}$ has to be A-orthonormalized only against $W_{k-1}$ and $W_{k-2}$ and then against itself. Finally, the approximate solution $x_k$ and the residual $r_k$ are updated, $x_k = x_{k-1}+W_k\alpha_k$ and $r_k = r_{k-1} - AW_k\alpha_k$, where $\alpha_k = W_k^tr_{k-1}$. This is the SRE-CG method.

However, in finite arithmetic there might be a loss of A-orthogonality at the $k^{th}$ iteration between the vectors of $Q_k = [W_1, W_2, ..., W_k]$. Hence, in SRE-CG2 $W_k = AW_{k-1}$ is A-orthonormalized against all $W_i$'s for $i=1,2,..,k-1$. 

The construction of $W_k$ matrix is independent from updating the approximate solution $x_k$. Thus it is possible to restructure the SRE-CG and SRE-CG2 algorithms by first computing $W_1$, $W_2,..., W_s$, and then updating $x_1$, $x_2,...,x_s$ as shown in Algorithm \ref{alg:restSRE-CG} and Algorithm \ref{alg:restSRE-CG2}.

The advantage of such reformulations (Algorithm \ref{alg:restSRE-CG} and Algorithm \ref{alg:restSRE-CG2}) is that the matrix $A$ is fetched once from memory per construction of $st$ $A$-orthonormal vectors; as opposed to fetching it $s$ times in the SRE-CG and SRE-CG2 algorithms. However, the number of messages and words sent in parallel is unchanged since the 2 corresponding algorithms perform the same operations but in a different order.

 \begin{algorithm}[h!]
\centering
\caption{ Restructured SRE-CG2  }
{\renewcommand{\arraystretch}{1.3}
\begin{algorithmic}[1]
\Statex{\textbf{Input:} $A$,  $n \times n$ symmetric positive definite matrix; $k_{max}$, maximum allowed iterations}
\Statex{\qquad \quad $b$,  $n \times 1$ right-hand side; $x_0$, initial guess; $\epsilon$, stopping tolerance; $s$, s-step}
\Statex{\textbf{Output:} $x_k$, approximate solution of the system $Ax=b$}
\State$r_0 = b - Ax_0$, $\rho_0 = ||r_0||_2$ , ${\rho} = \rho_0$, $k = 1$; 
\While {( ${\rho} > \epsilon {\rho_0}$ and $k < k_{max}$ )}
\If {($k==1$)}
\State A-orthonormalize $W_k = \mathscr{T}_0$,  and let $Q = W_k$ 
\Else 
\State A-orthonormalize $W_k = AW_{k-1}$ against $Q$
\State A-orthonormalize $W_k$ and let $Q = [Q \; W_k]$ 
\EndIf
\For {($i=1:s-1$)} 
\State A-orthonormalize $W_{k+i} = AW_{k+i-1}$ against $Q$
\State A-orthonormalize $W_{k+1}$ and let $Q = [Q \; W_{k+1}]$ 
\EndFor
\For {($i=k:k+s-1$)} 
\State $\tilde{\alpha} = (W_i^t r_{i-1})$, \;\; $x_i = x_{i-1} + W_i\tilde{\alpha} $  
\State $r_i = r_{i-1} - AW_i\tilde{\alpha} $ 
\EndFor
\State $k = k+s$, \;\; $\rho = ||r_{k-1}||_2$ 
\EndWhile
\end{algorithmic}}
\label{alg:restSRE-CG2}
\end{algorithm}

\begin{algorithm}[h!]
\centering
\caption{ Restructured SRE-CG }
{\renewcommand{\arraystretch}{1.3}
\begin{algorithmic}[1]
\Statex{\textbf{Input:} $A$,  $n \times n$ symmetric positive definite matrix; $k_{max}$, maximum allowed iterations}
\Statex{\qquad \quad $b$,  $n \times 1$ right-hand side; $x_0$, initial guess; $\epsilon$, stopping tolerance; $s$, s-step}
\Statex{\textbf{Output:} $x_k$, approximate solution of the system $Ax=b$}
\State$r_0 = b - Ax_0$, $\rho_0 = ||r_0||_2$ , ${\rho} = \rho_0$, $k = 1$; 
\While {( ${\rho} > \epsilon {\rho_0}$ and $k < k_{max}$ )}
%\State Let $j = k$
\If {($k==1$)}
\State A-orthonormalize $W_k = \mathscr{T}_0$%,  and let $W = W_j$ 
\Else 
\State  A-orthonormalize $W_{k} = AW_{k-1}$ against $W_{k-2}$ and $W_{k-1}$
\State A-orthonormalize $W_{k}$ %and let $W = W_{j}$  
\EndIf
\For {($i=1:s-1$)} 
\State A-orthonormalize $W_{k+i} = AW_{k+i-1}$ against $W_{k+i-2}$ and $W_{k+i-1}$
\State A-orthonormalize $W_{k+i}$ %and let $W = [W \; W_{j+i}]$ 
\EndFor
\For {($i=k:k+s-1$)} 
\State $\tilde{\alpha} = (W_i^t r_{i-1})$, \;\; $x_i = x_{i-1} + W_i\tilde{\alpha} $  
%\State $x_k = x_{k-1} + W_i\tilde{\alpha} $ 
\State $r_i = r_{i-1} - AW_i\tilde{\alpha} $ 
%\State $k = k+1$
\EndFor
\State $k = k+s$, \;\; $\rho = ||r_{k-1}||_2$ 
\EndWhile
\end{algorithmic}}
\label{alg:restSRE-CG}
\end{algorithm}
To reduce communication the inner for loops have to be replaced with a set of denser operations. Lines 3 till 12 of Algorithm \ref{alg:restSRE-CG2} can be viewed as a block Arnoldi A-orthonormalization procedure, whereas lines 4 till 13 of Algorithm \ref{alg:restSRE-CG} can be viewed as a truncated block Arnoldi A-orthonormalization procedure. As for the second loop, by updating $x_k$ and $r_k$ once, we obtain an s-step version. 

At the $k^{th}$ iteration of an s-step enlarged CG method, $st$ new basis vectors of $ \mathscr{K}_{ks,t}= span\{T(r_0), AT(r_0),...,A^{sk-1}T(r_0)\}$, are computed and stored in $V_{k}$, an $n \times st$ matrix. Since, $$ \mathscr{K}_{ks,t}=  \mathscr{K}_{(k-1)s,t} + span\{A^{s(k-1)}T(r_0), A^{s(k-1)+1}T(r_0)...,A^{sk-1}T(r_0)\},$$ then $Q_{ks} = [Q_{(k-1)s}, V_{k}]$, where $Q_{(k-1)s}$ is an $n \times (k-1)st$ matrix that contains the $(k-1)st$ vectors of $\mathscr{K}_{(k-1)s,t}$, and $Q_{ks}$ is $n \times kst$ matrix.

Then, $x_k = x_{k-1} + Q_{ks}\alpha_k \in \mathscr{K}_{ks,t}$, where  $\alpha_k = (Q_{ks}^tAQ_{ks})^{-1}(Q_{ks}^tr_{k-1})$ is defined by minimizing $\phi(x) = \frac{1}{2}x^tAx - b^tx$ over $x_0 + \mathscr{K}_{ks,t}$. As a consequence, $r_k = b-Ax_k = r_{k-1}-AQ_{ks}\alpha_k \in \mathscr{K}_{(k+1)s,t}$ satisfies the Petrov-Galerkin condition $r_k \perp \mathscr{K}_{ks,t}$, i.e. $r_k^ty = 0$ for all $y \in \mathscr{K}_{ks,t}$.  

In the s-step SRE-CG2 version, $Q_{ks}$ is A-orthonormalized  ($Q_{ks}^tAQ_{ks} = I$), then \vspace{+1mm}
$$\alpha_k = (Q_{ks}^tAQ_{ks})^{-1}(Q_{ks}^tr_{k-1}) = Q_{ks}^tr_{k-1}.$$\vspace{+2mm}
 But $r_{k-1} \perp \mathscr{K}_{(k-1)s,t}$, i.e. $r_{k-1}^ty = 0$ for all $y \in \mathscr{K}_{(k-1)s,t}$. Thus, 
 \begin{eqnarray}
\alpha_k  &=& Q_{ks}^tr_{k-1} = [Q_{(k-1)s}, V_{k}]^tr_{k-1} = [0_{(k-1)st \times n}; V_{k}^tr_{k-1}]\\
x_k &=& x_{k-1} + Q_{ks}\alpha_k = x_{k-1} +[Q_{(k-1)s}, V_{k}][0_{(k-1)st \times n}; V_{k}^tr_{k-1}]\\
 &=& x_{k-1} + V_{k}V_{k}^tr_{k-1} = x_{k-1} + V_{k}\tilde{\alpha}_k,
\end{eqnarray}   where $\tilde{\alpha}_k=V_{k}^tr_{k-1}$. Then, $r_k = r_{k-1} - AV_{k}\tilde{\alpha}_k$.

  \begin{algorithm}[h!]
\centering
\caption{ s-step SRE-CG2  }
{\renewcommand{\arraystretch}{1.3}
\begin{algorithmic}[1]
\Statex{\textbf{Input:} $A$,  $n \times n$ symmetric positive definite matrix; $k_{max}$, maximum allowed iterations}
\Statex{\qquad \quad $b$,  $n \times 1$ right-hand side; $x_0$, initial guess; $\epsilon$, stopping tolerance; $s$, s-step }
\Statex{\textbf{Output:} $x_k$, approximate solution of the system $Ax=b$}
\State$r_0 = b - Ax_0$, $\rho_0 = ||r_0||_2$ , $\rho = \rho_0$, $k = 1$; 
\While {( ${\rho} > \epsilon {\rho_0}$ and $k < k_{max}$ )}
\State Let $j = (k-1)s+1$
\If {($k==1$)}
\State A-orthonormalize $W_j = \mathscr{T}_0$, and  let $Q = W_j$ 
\Else 
\State A-orthonormalize $W_j = AW_{j-1}$ against $Q$
\State A-orthonormalize $W_j$,  and let $Q = [Q, \; W_j]$ 
\EndIf
\State Let $V = W_j$  
\For {($i=1:s-1$)} 
\State A-orthonormalize $W_{j+i} = AW_{j+i-1}$ against $Q$
\State A-orthonormalize $W_{j+i}$,  let  $V = [V, \; W_{j+i}]$ and $Q = [Q, \; W_{j+i}]$ 
\EndFor
\State $\tilde{\alpha} = V^t r_{k-1}$ 
\State $x_k = x_{k-1} + V\tilde{\alpha} $ 
\State $r_k = r_{k-1} - AV\tilde{\alpha} $  
\State $\rho = ||r_{k}||_2$, $k = k+1$  
\EndWhile
\end{algorithmic}}
\label{alg:sstepSRE-CG2}
\end{algorithm}

\begin{algorithm}[h!]
\centering
\caption{ s-step SRE-CG }
{\renewcommand{\arraystretch}{1.3}
\begin{algorithmic}[1]
\Statex{\textbf{Input:} $A$,  $n \times n$ symmetric positive definite matrix; $k_{max}$, maximum allowed iterations}
\Statex{\qquad \quad $b$,  $n \times 1$ right-hand side; $x_0$, initial guess; $\epsilon$, stopping tolerance; $s$, s-step }
\Statex{\textbf{Output:} $x_k$, approximate solution of the system $Ax=b$}
\State$r_0 = b - Ax_0$, $\rho_0 = ||r_0||_2$ , $\rho = \rho_0$, $k = 1$; 
\While {( ${\rho} > \epsilon {\rho_0} $ and $k < k_{max}$ )}
\State Let $j = (k-1)s+1$
\If {($k==1$)}
\State A-orthonormalize $W_j = \mathscr{T}_0$,  and let $V = W_j$ 
\Else 
\State A-orthonormalize $W_{j} = AW_{j-1}$ against $W_{j-2}$ and $W_{j-1}$
\State A-orthonormalize $W_{j}$ and let $V = W_{j}$  
\EndIf
\For {($i=1:s-1$)} 
\State A-orthonormalize $W_{j+i} = AW_{j+i-1}$ against $W_{j+i-2}$ and $W_{j+i-1}$
\State A-orthonormalize $W_{j+i}$ and let $V = [V \; W_{j+i}]$ 
\EndFor
\State $\tilde{\alpha} = (V^t r_{k-1})$ 
\State $x_k = x_{k-1} + V\tilde{\alpha} $ 
\State $r_k = r_{k-1} - AV\tilde{\alpha} $  
\State $\rho = ||r_{k}||_2$, $k = k+1$  
\EndWhile
\end{algorithmic}}
\label{alg:sstepSRE-CG1}
\end{algorithm}

In Algorithm \ref{alg:sstepSRE-CG2}, the $st$ new vectors are computed similarly to Algorithm \eqref{alg:restSRE-CG2}, where $t$ vectors are computed at a time ($W_j$), A-orthonormalized against all the previously computed vectors using CGS2 A-orthonormalization method \cite{sophiethesis}, and finally A-orthonormalized using A-CholQR \cite{A-ortho} or Pre-CholQR \cite{A-ortho2, sophiethesis}. At the $k^{th}$ s-step iteration, all the $kst$ vectors have to be stored in $Q_{ks}$.  

Note that in exact arithmetic, at the $k^{th}$ s-step iteration, the A-orthonormalization of $W_j$ for $j\geq (k-1)s+1$, against $\tilde{Q} = [W_1, W_2, W_3, ... W_{j-1} ]$ can be summarized as follows, 
where $Q_{(k-1)s} = [W_1, W_2, W_3, ... W_{(k-1)s}]$,
%$[Q_{(k-1)s}, W_{(k-1)s+1}, \; ... \; W_{j-1}]$ 

\begin{eqnarray*}
W_j &=& AW_{j-1} - \tilde{Q}\tilde{Q}^tA(AW_{j-1}) \\
&=& AW_{j-1} - \tilde{Q}[W_1^tA(AW_{j-1});\, W_2^tA(AW_{j-1});\, ... \,;\, W_{j-1}^tA(AW_{j-1})] \\
&=& AW_{j-1} - \tilde{Q}[0;\, 0;\, ... \,;\,0;\, W_{j-2}^tA(AW_{j-2});\,W_{j-1}^tA(AW_{j-1})] \\
&=& AW_{j-1} - W_{j-1}W_{j-1}^tA(AW_{j-1}) - W_{j-2}W_{j-2}^tA(AW_{j-1}),\\ \vspace{-2mm}
\end{eqnarray*} 

\noindent since $(AW_{i})^tAW_{j-1} = 0$ for all $i<j-2$ by the A-orthonormalization process. %of the basis vectors of $\mathscr{K}_{t,k-1}$. 
This version (Algorithm \ref{alg:sstepSRE-CG1}) is called the s-step short recurrence enlarged conjugate gradient (s-step SRE-CG), where only  the last $zt$ computed vectors ($z = max(s,3)$) are stored, and  every $t$ vectors $W_j$ are A-orthonormalize against the previous $2t$ vectors $W_{j-2}$ and $W_{j-1}$ for $j>2$. As for $x_{k}$, and $r_{k}$, they are defined as in the s-step SRE-CG2 method.

For $s=1$, Algorithms \ref{alg:sstepSRE-CG2} and \ref{alg:sstepSRE-CG1} are reduced to the SRE-CG2 and SRE-CG methods, where the total number of messages sent in parallel is $6klog(t)$, assuming that the number of processors is set to $t$ and that the methods converge in $k$ iterations. Note that, more words are sent in the SRE-CG2 Algorithm \ref{alg:sstepSRE-CG2}, than in SRE-CG Algorithm \ref{alg:sstepSRE-CG1}, due to the A-orthonormalization procedure \cite{EKS}.

 For $s>1$, Algorithms \ref{alg:sstepSRE-CG2} and \ref{alg:sstepSRE-CG1} send $(s-1)log(t)$ less messages and words per s-step iteration, than Algorithm \ref{alg:restSRE-CG2} and \ref{alg:restSRE-CG}, assuming we have $t$ processors with distributed memory. This communication reduction is due to the computation of one $\alpha$ which consists of an $n\times st$ matrix vector multiplication, rather than $s$ computations of $n\times t$ matrix vector multiplications. Thus the total number of messages sent in parallel in Algorithms \ref{alg:sstepSRE-CG2} and \ref{alg:sstepSRE-CG1} is $5s{k}_{s}log(t) + {k}_{s}log(t)$, where $k_s$ is the number of $s-step$ iterations needed till convergence. Similarly to the case of $s=1$, the s-step SRE-CG2 Algorithm \ref{alg:sstepSRE-CG2} sends more words than the s-step SRE-CG Algorithm \ref{alg:sstepSRE-CG1}. 

Algorithms \ref{alg:sstepSRE-CG2} and \ref{alg:sstepSRE-CG1} will converge in $k_s$ iterations, where $k_s \geq \ceil{\frac{k}{s}}$, and $k$ is number of iterations needed for convergence for $s=1$. In exact arithmetic, every s-step iteration of Algorithms \ref{alg:sstepSRE-CG2} and \ref{alg:sstepSRE-CG1} is equivalent to $s$ iteration of the  SRE-CG2 and SRE-CG Algorithms, respectively. However, they might not be equivalent in finite arithmetic due to the loss of A-orthogonality of the $Q_{ks}$ matrix.  
At the first iteration of Algorithms \ref{alg:sstepSRE-CG2} and \ref{alg:sstepSRE-CG1}, 
\begin{eqnarray}
x_1 &=& x_{0} + V_1\tilde{\alpha}_1 = x_{0} + V_1(V_1^tr_0) \nonumber \\
&=& x_{0} + [W_1 W_2 ... \; Ws][W_1 W_2 ... \; Ws]^t r_0 \nonumber \\
&=& x_{0} + \sum_{i=1}^{s} W_iW_i^t r_0 \label{x1}
\end{eqnarray}
For $s=3$, \begin{eqnarray} x_1 &=& x_{0} + W_1W_1^t r_0 + W_2W_2^t r_0 + W_3W_3^t r_0.\nonumber \end{eqnarray}
On the other hand, after 3 iterations of the SRE-CG2 and SRE-CG Algorithms, the solution $x_3$ is:
\begin{eqnarray}
x_1 &=& x_{0} + W_1(W_1^tr_0) \nonumber \\
r_1 &=& r_0 - AW_1W_1^tr_0 \nonumber \\
x_2 &=& x_1 + W_2(W_2^tr_1) = x_{0} + W_1(W_1^tr_0) + W_2W_2^t(r_{0} - AW_1W_1^tr_0) \nonumber\\
&=& x_{0} + W_1W_1^tr_0 + W_2W_2^tr_{0} - W_2(W_2^tAW_1)W_1^tr_0 \nonumber\\
r_2 &=& r_0 - AW_1W_1^tr_0 - AW_2W_2^tr_{0} + AW_2(W_2^tAW_1)W_1^tr_0\nonumber \\
x_3 &=& x_2 + W_3W_3^tr_2 = x_{0} + W_1W_1^tr_0 + W_2W_2^tr_{0} - W_2(W_2^tAW_1)W_1^tr_0  \nonumber \\
&& + W_3W_3^t( r_0 - AW_1W_1^tr_0 - AW_2W_2^tr_{0} + AW_2W_2^tAW_1W_1^tr_0 ) \nonumber \\
&=& x_{0} + W_1W_1^tr_0 + W_2W_2^tr_{0} +  W_3W_3^tr_0  - W_2(W_2^tAW_1)W_1^tr_0 \nonumber \\
&& - W_3(W_3^tAW_1)W_1^tr_0  - W_3(W_3^tAW_2)W_2^tr_{0} + W_3(W_3^tAW_2)(W_2^tAW_1)W_1^tr_0 \nonumber 
\end{eqnarray}

For $s>3$, more terms with $W_j^tAW_i$ will be added. Assuming that $W_j^tAW_i = 0$ for all $j<i$, then the obtained $x_s$ in the SRE-CG2 and SRE-CG Algorithms, is equivalent to $x_1$ \eqref{x1} in the s-step  SRE-CG2 and s-step SRE-CG Algorithms. Similarly, under the same assumptions, $x_{is}$ in the SRE-CG2 and SRE-CG Algorithms, is equivalent to $x_i$ in the s-step  SRE-CG2 and s-step SRE-CG Algorithms. In case for some $j<i$, $W_j^tAW_i \neq 0$, then all the subsequent s-step solutions will not be equal to the corresponding SRE-CG2 and SRE-CG solutions.

%%%%%%%%%%%%%%%%%%%%%%
 Assuming that the s-step versions converge in $k_s = \ceil{\frac{k}{s}}$ iterations, %i.e.   $\floor{\frac{k}{s}} + 1$ iterations,
  then,   $5{k}log(t) + \frac{k}{s}log(t)$ messages are sent in parallel.
Hence, by merging $s$ iterations of the enlarged CG methods for some given value $t$, communication is reduced by a total of at most $(s-1)log(t)k_s = \frac{s-1}{s} log(t) k$ less messages and words. 

Theoretically, it is possible to further reduce communication by replacing the block Arnoldi A-orthonormalization (Algorithm \ref{alg:sstepSRE-CG2} lines 3-12) and the truncated block Arnoldi A-orthonormalization (Algorithm \ref{alg:sstepSRE-CG1} lines 4-13) with a communication avoiding kernel that first computes the $st$ vectors and then A-orthonormalize them against previous vectors and against themselves, as summarized in Algorithm \ref{alg:CA-Arnoldi}. These methods are called communication avoiding SRE-CG2 (CA SRE-CG2) and communication avoiding SRE-CG (CA SRE-CG2) respectively. For the first iteration ($k=1$), $W_{j-1} = [T(r_0)] = \mathscr{T}_0$ in Algorithm \ref{alg:CA-Arnoldi}. 
 \begin{algorithm}[h!]
\centering
\caption{CA-Arnoldi A-orthonormalization}
{\renewcommand{\arraystretch}{1.3}
\begin{algorithmic}[1]
\Statex{\textbf{Input:} $W_{j-1}$,  $n \times t$ matrix}; $k$, iteration% containing the last $t$ computed vectors }
\Statex{\qquad \quad $Q$,  $n \times m$ matrix, $m = (s+1)t$ in CA SRE-CG and $m = kst$ in CA SRE-CG2}
\Statex{\textbf{Output:} $V$, the $n \times st$ matrix containing the A-orthonormalized $st$ computed vectors}
\State \textbf{if}{ ($k==1$)} \textbf{then} $W_j = W_{j-1}$, and $V = W_j$
\State \textbf{else} $W_j = AW_{j-1}$, and $V = W_j$ 
\State \textbf{end if}
\For {($i=1:s-1$)} 
\State Let $W_{j+i} = AW_{j+i-1}$ 
\State Let $V = [V \; W_{j+i}]$ 
\EndFor
\State \textbf{if}{ ($k>1$)}  \textbf{then} A-orthonormalize $V$ against $Q$ \textbf{end if}
\State A-orthonormalize $V$
\end{algorithmic}}
\label{alg:CA-Arnoldi}
\end{algorithm}

In s-step SRE-CG, the $st$  vectors are computed and A-orthonormalized against the previous $2t$ vectors, $t$ vectors at a time.
But in the case of CA SRE-CG, the $st$ vectors are all computed before being A-orthonormalized.
Thus, it is not sufficient to just A-orthonormalize the $st$ computed vectors against the last $2t$ vectors. Instead, the $st$ computed vectors should be A-orthonormalized  against the last $(s+1)t$ vectors.

Assuming that $Q = Q_{(k-1)s} = [W_1, W_2, W_3, ... W_{(k-1)s}]$ is A-orthonormal, then for all   $i+l<j$ where $j=(k-1)s$, we have that $$(A^iW_{l})^tAW_{j}=W_{l}^tA(A^iW_{j}) = 0.$$ After computing $W_{j+i} = A^i W_{(k-1)s} = A^iW_j$ for $i = 1, .., s$, the A-orthonormalization is summarized as follows:
\begin{eqnarray*}
W_{j+i} &=& A^iW_{j} - {Q}{Q}^tA(A^iW_{j}) \\
&=& A^iW_{j} - {Q}[W_1^tA(A^iW_{j});\, W_2^tA(A^iW_{j});\, ... \,;\, W_{j}^tA(A^iW_{j})] \\
%&=& A^iW_{j} - {Q}[0;...;0; W_{j-i}^tA(A^iW_{j});\, W_{j-i+1}^tA(A^iW_{j});\, ... \,;\, W_{j}^tA(A^iW_{j})] \\
&=& A^iW_{j} - \sum_{l=1}^{j}W_l W_{l}^tA(A^iW_{j}) = A^iW_{j} - \sum_{l=j-i}^{j}W_l W_{l}^tA(A^iW_{j}).% \vspace{-2mm}
\end{eqnarray*} 
This implies that $W_{j+1}$ should be A-orthonormalized against the last $2t$ vectors $W_{j-1},$ and $W_{j}$. Whereas, $W_{j+s}$ should be A-orthonormalized against the last $(s+1)t$ vectors 

\noindent $W_{j-s},W_{j-s+1},..., W_{j}$. And in general, $W_{j+i}$ should be A-orthonormalized against the last $(i+1)t$ vectors. To reduce communication, in CA-SRE-CG  all of the $st$ computed vectors, $W_{j+1}, W_{j+2}, ..., W_{j+s}$, are A-orthonormalized against the previous $(s+1)t$ vectors.
 
 Given that we are computing the monomial basis, the $st$ computed vectors might  be linearly dependent, which leads to a numerically unstable basis. The numerical stability and convergence of such communication avoiding and s-step versions is discussed in section \ref{sec:SRNum}.

 \subsection{s-step MSDO-CG}   
 
 The MSDO-CG method \cite{EKS} computes $t$ search directions at each iteration $k$, $P_k =  \mathscr{T}_{k-1} + P_{k-1}diag(\beta_k)$ where $P_0 = \mathscr{T}_{0}$ and $\mathscr{T}_{i}$ is the matrix containing the $t$ vectors of $T(r_i)$. Then, $P_k$ is A-orthonormalized against all $P_i$'s ($i<k$), and used to update $x_k = x_{k-1} + P_k\alpha_k$ and $r_k = r_{k-1} - AP_k\alpha_k$, where $\alpha_k  =  P_k^tr_{k-1}$. This procedure is interdependent since we can not update $P_k$ without $r_{k-1}$, and we can not update $r_{k-1}$ without $P_{k-1}$. Thus, to build an s-step version we need to split the computation of $P_k$ and the update of $x_k$, which is not possible. For that purpose we introduce a modified version of MSDO-CG where we build a modified Enlarged Krylov basis rather than computing search directions. 
 
  As discussed in \cite{EKS}, the vectors of $P_k$ belong to the Enlarged Krylov subspace $$ \mathscr{K}_{k,t} = span \{ T(r_0), AT(r_0), .., A^{k-1}T(r_0)\}.$$ Moreover, the vectors of $P_k$ belong to the modified Enlarged Krylov subspace \begin{equation}
    \overline{\mathscr{K}}_{k,t} = span\{ T(r_0), T(r_1), T(r_2), ..., T(r_{k-1}) \}. \label{modeks}\end{equation}
 
 In general, we define the modified Enlarged Krylov subspace for a given $s$ value as follows 
 \begin{eqnarray}
   \overline{\mathscr{K}}_{k,t,s}  = span & \{& T(r_0), AT(r_0), ..., A^{s-1}T(r_0), \nonumber \\ 
   & &T(r_1), AT(r_1), ..., A^{s-1}T(r_1), \nonumber \\
   & &T(r_2), AT(r_2), ..., A^{s-1}T(r_2), \nonumber \\
   & & \vdots  \nonumber \\
   & & T(r_{k-1}), AT(r_{k-1}), ..., A^{s-1}T(r_{k-1}) \nonumber  \}
\end{eqnarray} 
Note that for $s=1$, the  modified Enlarged Krylov subspace becomes $\overline{\mathscr{K}}_{k,t}$ defined in \eqref{modeks}. % = span\{ T(r_0), T(r_1), T(r_2), ..., T(r_{k-1}) \}.$$ 
Moreover, for $t=s=1$, the  modified Enlarged Krylov subspace becomes $$  \overline{\mathscr{K}}_{k,1,1} = span\{ r_0, r_1, r_2, ..., r_{k-1} \}, $$ where the Krylov subspace $\mathcal{K}_k = span\{ r_0, Ar_0, A^2r_0, ..., A^{k-1}r_{0} \} = span\{ r_0, r_1, r_2, ..., r_{k-1}\}$.

Similarly to the Enlarged Krylov subspace $\mathscr{K}_{ks,t} $, the modified Enlarged Krylov subspace $\overline{\mathscr{K}}_{k,t,s}$ is of dimension at most $kst$. 
\begin{theorem}\label{subk22}
The Krylov subspace $\mathcal{K}_k$ is a subset of the modified enlarged Krylov subspace $ \overline{\mathscr{K}}_{k,t,1}$ ($\mathcal{K}_k \subset \overline{\mathscr{K}}_{k,t,1}$).
\end{theorem}

\begin{proof}
 Let $y \in  \mathcal{K}_k$ where $\mathcal{K}_k = span\{ r_0, Ar_0,.., A^{k-1}r_0\} = span\{ r_0, r_1, r_2, ..., r_{k-1}\}$.  Then,  
 %\vspace*{-3mm}
 \begin{eqnarray}y&=&  \sum_{j=0}^{k-1} a_jr_j = \sum_{j=0}^{k-1} a_j\mathscr{T}_j*\mathbbm{1}_t = \sum_{j=0}^{k-1} \sum_{i=1}^t a_jT_i(r_j) \in \overline{\mathscr{K}}_{k,t,1} \nonumber \end{eqnarray} 
 since $r_j =  \mathscr{T}_j*\mathbbm{1}_t = [T_1(r_j) \,T_2(r_j)\, ....\, T_t(r_j)]*\mathbbm{1}_t$ , 
 where $\mathbbm{1}_t$ is a $t \times 1$ vector of ones, and $\mathscr{T}_j = [T_1(r_j) \,T_2(r_j)\, ....\, T_t(r_j)] = [T(r_j)]$ is the matrix containing the $t$ vectors of $T(r_j)$.
\end{proof}\vspace{+2mm}
 
  Then one possible s-step reformulation of MSDO-CG would be to compute basis vectors of $\overline{\mathscr{K}}_{k,t,s}$ and use them to update the solution and the residual, similarly to the s-step SRE-CG versions. At iteration $k$ of the s-step MSDO-CG method (Algorithm \ref{alg:sstepMSDO-CG}), the $st$ vectors \vspace{+1mm}
$$T(r_{k-1}), AT(r_{k-1}), ..., A^{s-1}T(r_{k-1}) \vspace{+1mm}$$ 
are computed and A-orthonormalized similarly to the s-step SRE-CG2 method, and stored in the $n \times st$ matrix $V_{k}$. Then, these $st$ A-orthonormalized vectors are used to define $\tilde{\alpha}_k = V_{k}^t r_{k-1}$ and update $x_k = x_{k-1} + V_{k}\tilde{\alpha}$ and $r_k = r_{k-1} - AV_{k}\tilde{\alpha}$.

  \begin{algorithm}[h!]
\centering
\caption{ s-step MSDO-CG  }
{\renewcommand{\arraystretch}{1.3}
\begin{algorithmic}[1]
\Statex{\textbf{Input:} $A$,  $n \times n$ symmetric positive definite matrix; $k_{max}$, maximum allowed iterations}
\Statex{\qquad \quad $b$,  $n \times 1$ right-hand side; $x_0$, initial guess; $\epsilon$, stopping tolerance; $s$, s-step }
\Statex{\textbf{Output:} $x_k$, approximate solution of the system $Ax=b$}
\State$r_0 = b - Ax_0$, $\rho_0 = ||r_0||_2$ , ${\rho} = \rho_0$,  $k = 1$; 
\While {( ${\rho} > \epsilon {\rho_0} $ and $k < k_{max}$ )}
\If {($k==1$)}
\State A-orthonormalize $W_1 = \mathscr{T}_0$,   let $V = W_1$ and $Q = W_1$ 
\Else 
\State A-orthonormalize $W_1 = \mathscr{T}_{k-1}$ against $Q$
\State A-orthonormalize $W_1$,  let $V = W_1$ and $Q = [Q \; W_1]$ 
\EndIf
\For {($i=1:s-1$)} 
\State A-orthonormalize $W_{i+1} = AW_i$ against $Q$
\State A-orthonormalize $W_{i+1}$,  let $V = [V \; W_{i+1}]$ and $Q = [Q \; W_{i+1}]$ 
\EndFor
%\State $W = Q(:, (k-1)st+1:kst)$
\State $\tilde{\alpha} = (V^t r_{k-1})$ 
\State $x_k = x_{k-1} + V\tilde{\alpha} $ 
\State $r_k = r_{k-1} - AV\tilde{\alpha} $  
\State $\rho = ||r_{k}||_2$, $k = k+1$  
\EndWhile
\end{algorithmic}}
\label{alg:sstepMSDO-CG}
\end{algorithm}

Note that for $s=1$, the s-step MSDO-CG method is reduced to a modified version of MSDO-CG. Although the s-step MSDO-CG method for $s=1$ is different algorithmically than the MSDO-CG method, but they converge in the same number of iterations as shown in section \ref{sec:SRNum} due to their theoretical equivalence. Moreover, each iteration of the s-step MSDO-CG method with $s>1$ is not equivalent to $s$ iterations of the modified version of MSDO-CG, since the constructed bases of $\overline{\mathscr{K}}_{k,t,s}$ and $\overline{\mathscr{K}}_{ks,t,1}$ are different. For example, in the second iteration of s-step MSDO-CG ($k = 2$), $T(r_1), AT(r_1),...,A^{s-1}T(r_1)$ are computed. Whereas in the second $s$ iterations of the modified version of MSDO-CG ($ks = 2s$), the vectors $T(r_s), T(r_{s+1}), ..., T(r_{2s-1})$ are computed.

The communication avoiding MSDO-CG differs from the s-step version (Algorithm \ref{tab:MSDOCG}) in the basis construction where at the $k^{th}$ iteration the $st$ vectors $T(r_{k-1}), AT(r_{k-1}),...,A^{s-1}T(r_{k-1})$ are first computed, and then A-orthonormalized against previous vectors and against themselves. Thus the communication avoiding MSDO-CG algorithm is  Algorithm \ref{tab:MSDOCG} with the replacement of lines 3-12 by the  CA-Arnoldi A-orthonormalization Algorithm \ref{alg:CA-Arnoldi}. However, Algorithm \ref{alg:CA-Arnoldi} is slightly modified,  where in line 2 $W_j = W_{j-1}$ rather than $W_j = AW_{j-1}$, with $W_{j-1} =\mathscr{T}_{k-1} = [T(r_{k-1})]$ for $k\geq 1$.

The advantage of building the modified Enlarged Krylov subspace basis is that at iteration $k$,  each of the $t$ processors can compute the $s$ basis vectors
$$T_i(r_{k-1}), AT_i(r_{k-1}), A^2T_i(r_{k-1}),..., A^{s-1}T_i(r_{k-1})  $$
 independently, where $T_i(r_{k-1})$ is the projection of the vector $r_{k-1}$ on the $i^{th}$ domain of the matrix $A$, i.e. a vector of all zeros except at $n/t$ entries that correspond to the $i^{th}$ domain. Thus, there is no need for communication avoiding kernels, since processor $i$ needs a part of the matrix $A$ and a part of the vector $r_{k-1}$ to compute the $s$ vectors. As a consequence, assuming that enough memory is available, then any preconditioner can be applied to the CA MSDO-CG since the Matrix Powers Kernel is not used to compute the basis vectors, as discussed in section \ref{sec:pm}.
 
 \section{Numerical Stability and Convergence}\label{sec:SRNum}
We compare the convergence behavior of the different introduced s-step enlarged CG versions and their communication avoiding versions for solving the system $Ax = b$ using different number of partitions ($t = 2, 4, 8, 16, 32$, and $64$ partitions) and different number of $s$-values (1, 2, 3, 4, 5, 8, 10).  Similarly to the enlarged CG methods \cite{EKS}, the matrix $A$ is first reordered using Metis's kway partitioning \cite{metis} that defines the $t$ subdomains. Then $x$ is chosen randomly using MATLAB's rand function and the right-hand side is defined as $b = Ax$. The initial iterate is set to $x_0 = 0$, and the stopping criteria tolerance is set to $tol = 10^{-8}$ for all the matrices, except {\poisso} ($tol = 10^{-6}$).

The  characteristics of the test matrices are summarized in Table \ref{tab:testmatrices}.  The  {\poisso} matrix is a block tridiagonal matrix obtained from Poisson's equation using MATLAB's  ``gallery(`poisson',100)''. 
The remaining matrices, referred to as {\nho}, {\skyo}, {\skyto}, and {\anio}, arise from different boundary value problems of convection diffusion equations, and generated using FreeFem++ \cite{freefem}. For a detailed description of the test matrices, refer to \cite{EKS}.
 
%\vspace{-4mm}
 \begin{table}[h!]
\centering
{\renewcommand{\arraystretch}{1.4}\footnotesize
\caption{The test matrices}
\begin{tabular}{|c|c|c|c|c|}
\hline
\textbf{Matrix} & \textbf{Size} & \textbf{Nonzeros }  & \textbf{2D/3D} & \textbf{Problem}\\ \hline 
{\poiss} & 10000 & 49600 & 2D & Poisson equations \\% &Yes&Achdou \cite{achdou} \\ 
{\nh} & 10000 &49600  & 2D & Boundary value \\% &Yes&Achdou \cite{achdou} \\ 
{\sky} & 10000 & 49600 & 2D & Boundary value \\% &Yes&Achdou \cite{achdou} \\ 
{\skyt} & 8000 & 53600  & 3D &  Skyscraper \\% &Yes&Achdou \cite{achdou} \\ 
{\ani} & 8000& 53600  & 3D & Anisotropic Layers \\
\hline
\end{tabular}\label{tab:testmatrices}}
\end{table}

The first phase in all the discussed algorithms, is building the A-orthonormal basis by A-orthonormalizing a set of vectors against previous vectors using Classical Gram Schmidt A-orthonormalization (CGS), CGS2, or MGS, and then against themselves using A-CholQR \cite{A-ortho} or Pre-CholQR \cite{A-ortho2}.
As discussed in \cite{sophiethesis}, the combinations CGS2+A-CholQR  and CGS2+Pre-CholQR are both numerically stable and require less communication. In this paper, we test the introduced methods using CGS2 (Algorithm 18 in \cite{sophiethesis}), A-CholQR (Algorithm 21 in \cite{sophiethesis}), and Pre-CholQR (Algorithm 23 in \cite{sophiethesis}). Based on the performed testing, Pre-CholQR is numerically more stable than A-CholQR. However, for most of the tested cases, the versions with CGS2+A-CholQR or CGS2+Pre-CholQR A-orthonormalization converge in the same number of iterations.
 
 In Table \ref{tab:SRECG} we compare the convergence behavior of the different SRE-CG versions with respect to number of partitions $t$ and the $s$ values. The restructured SRE-CG is a reordered version of SRE-CG, where the same operations of $s$ iterations are performed, but in a different order. In addition, the check for convergence is done once every $s$ iterations. 
Thus the restructured SRE-CG Algorithm \ref{alg:restSRE-CG} converges in $s * k_s$ iterations. In Table \ref{tab:SRECG}, $k_s$ is shown rather than $s*k_s$, for comparison purposes with the s-step versions. Moreover, for $s=1$, the restructured SRE-CG is reduced to SRE-CG, and it converges in $k$ iterations. 

In case in Algorithm \ref{alg:restSRE-CG}, the second inner for loop is replaced by  a while loop with a check for convergence ($||r_{i-1}||_2> \epsilon ||r_0||_2$), then the Algorithm converges in exactly $s*\ceil{\frac{k}{s}}$ iterations, where SRE-CG converges in $k$ iterations and $k \leq s*\ceil{\frac{k}{s}} \leq  k + s-1$.

\begin{table}[h!]
\setlength{\tabcolsep}{2pt}
\caption{\label{tab:SRECG} Comparison of the convergence of different SRE-CG versions  (restructured SRE-CG, s-step SRE-CG, and CA SRE-CG with Algorithm \ref{alg:CA-Arnoldi2}), with respect to number of partitions $t$ and $s$ values. }
  \centering
  %\tabsize    
    \renewcommand{\arraystretch}{1.2}
  \begin{tabular}{||c||c||c||c|c|c|c|c|c|c||c|c|c|c|c|c||c|c|c||}
    \cline{4-19} 
     \cline{4-19}
    \multicolumn{1}{c}{}& \multicolumn{1}{c}{}& \multicolumn{1}{c||}{} & \multicolumn{7}{c||}{\multirow{2}{*}{\bf Restructured SRE-CG }} & \multicolumn{6}{c||}{\multirow{2}{*}{\bf s-step SRE-CG }}  & \multicolumn{3}{c||}{\multirow{2}{*}{\bf CA SRE-CG}}\\
        \multicolumn{1}{c}{}& \multicolumn{1}{c}{}& \multicolumn{1}{c||}{} & \multicolumn{7}{c||}{\multirow{2}{*}{}} & \multicolumn{6}{c||}{\multirow{2}{*}{}}  & \multicolumn{3}{c||}{{\bf }}\\
    \cline{2-19}
    \multicolumn{1}{c||}{} & \multicolumn{1}{c||}{\bf CG} & \backslashbox{$\bf t$}{$\bf s$} & \bf 1 &\bf  2  &\bf   3 &\bf  4  &\bf  5 &\bf  8 &\bf  10  &\bf  2 &\bf  3 &\bf  4  &\bf  5 &\bf  8 &\bf  10  &\bf  2 &\bf  3 & \bf 4 \\
    \hline \hline
    \multirow{6}{*}{\rotatebox[origin=c]{90}{\bf \poiss}}
   &\multirow{6}{*}{195} &\bf 2   &193 & 97&65 &49  & 39& 25&	20& 97&65 &49  & 39 &25&	20& 97 &65&49 \\
    \cline{3-19}
   & & \bf 4    &153&	77&	51&	39&	31& 20&	16& 77	&51&39&31& 20&	16& 77&51&39 \\
     \cline{3-19}
    &&\bf  8    &123&	62	&41	&31&	25&16&	13&	62	&41	&31	&25&16&	13&		62&	41	&31   \\    
     \cline{3-19}
    &&\bf 16   &95&	48&	32	&24&	19&12&	10&	48&	32	&24&	19&12&	10&	48&	32	&24 \\
      \cline{3-19}
   & &\bf 32   &70&	35&	24&	18&	14&9&	7&	35&	24&	18&	14&	9&	7&35&	24&	18  \\
     \cline{3-19}
    &&\bf 64    &52&	26&	18&	13&	11&7	&6&	26&	18&	13&	11&	7	&6& 26&	18&	13
  \\
   \hline
  \hline
      \multirow{6}{*}{\rotatebox[origin=c]{90}{\nh} }
   &\multirow{6}{*}{259} &\bf 2   & 245&	123&	82&	62&	49& 31&	25&	123&	82&	62&	49& 31&	25&	123&	82&	62 \\
    \cline{3-19}
   & &\bf 4 & 188&	94&	63&	47&	38&	24&	19& 94&	63	&47	&38& 24&	19&	94	&63	&47\\
     \cline{3-19}
    &&\bf 8    &149	&75&	50&	38&	30&	19&	15& 75&	50&	38&	30&	19&	15& 75&	50&	38   \\    
     \cline{3-19}
    &&\bf 16   & 112&	56&	38&	28&	23&14&	12&	56&	38&	28&	23& 14&	12&	56&	38&	28 \\
      \cline{3-19}
   & &\bf 32   &  82&	41&	28&	21&	17& 11&	9& 41& 28&	21&	17&	11&	9& 41&	28&	21  \\
     \cline{3-19}
     
    &&\bf 64    &60&	30&	20&	15&	12&	8&	6& 30&	20&	15&	12&	8&	6&  30&	20&	15
  \\
   \hline
  \hline
      \multirow{6}{*}{\rotatebox[origin=c]{90}{\sky}}
   &\multirow{6}{*}{5951} &\bf 2   & 5526&	2763&	1842&	1395&	1116& 708 & 558&	2763&	1842&	1395&	1116& 708 & 558&	 2793&	1854&	x \\
    \cline{3-19}
   & &\bf 4 & 4526&	2263&	1521&	1141&	913&571&457&		2263&	1521&	1141&	913&571&457&	  2328&	1575 &	x\\
     \cline{3-19}
    &&\bf 8    &2843&	1423&	949	&712&	575& 356&288&	1423&	949	&712&	575& 356&334	&  1405&	973&	x   \\    
     \cline{3-19}
    &&\bf 16   & 1770&	885&590&	450&	354&225&177&885&590&	450&	354&225&183&	  910&605	&x \\
      \cline{3-19}
   & &\bf 32   &  999&	500&	333&	250&	200&125&100&	500&	333&	250&	200&125&x&  	492&	340	&x   \\
     \cline{3-19}
     
    &&\bf 64    &507&	255&	169&	128&	102&64&51&255&169&128&102&x&x& 257&179&x
  \\
   \hline
  \hline
      \multirow{6}{*}{\rotatebox[origin=c]{90}{\skyt }}
  &\multirow{6}{*}{902} &\bf 2&829&435&290&218&174&109&87&435&290&218&174&109&87& 426&	285&	x \\
   \cline{3-19}
  & &\bf 4 & 745&	382&255&	191&	149&98&78&382&255&191&149&142&473& 373&251&x\\
    \cline{3-19}
    &&\bf 8&590&295&197&148&118&74&59&295&197&148&118&110&199& 294&198&	x   \\    
     \cline{3-19}
   &&\bf 16   & 436	&218&146&109&89&56&45& 218&146&109&89&58&74& 223&150&x\\
     \cline{3-19}
   & &\bf 32   &  279	&142&93&	71&57&36&29&142&93&71&57&x&x& 140&97&x   \\
     \cline{3-19}    
    &&\bf 64    &157	&79&	53&	40&	32&20&16	&79&	53&	40&	32&314&250&	 78&	54&x\\
   \hline
  \hline  
       \multirow{6}{*}{\rotatebox[origin=c]{90}{\ani}}
   &\multirow{6}{*}{4146} &\bf 2 &4005&2030&1335&1015&801&510&406&2030&1335&1015&801&510&406&  1985&1346&x\\
    \cline{3-19}
   & &\bf 4 & 3570&	1785&	1190&	909&	714&464&357&1785&1190&909&714&464&357&   1776&	1201&x\\
     \cline{3-19}
    &&\bf 8    &3089&1612&	1075&806&645&403&325&1612&1075&x&x&x&x&1548&1070&	x  \\    
     \cline{3-19}
    &&\bf 16   & 2357&1219&815&	610&	488&305&244&	1219&815&x&x&x&x& 1169&800&x\\
      \cline{3-19}
   & &\bf 32   & 1640&820&552&410&328&205&164&820&552&1686&2729&1804&1499&816&549&x   \\
     \cline{3-19}    
    &&\bf 64    &928	&464&315&232&189&116&95&464&315&	792&	777&	492&501& 453&	316&x\\
  \hline
 \hline
  \end{tabular}\vspace{-5mm}
\end{table}

 Thus, it is expected that the restructured SRE-CG (Algorithm \ref{alg:restSRE-CG}) converges in $k_1$ iterations, where $k_1 \geq s*\ceil{\frac{k}{s}}$. In case SRE-CG converges in $k$ iterations, and $k$ is divisible by $s$, then Algorithm \ref{alg:restSRE-CG} converges in exactly $k_1 = s*\ceil{\frac{k}{s}} = k$ iterations. On the other hand, if $k$ is not divisible by $s$, then Algorithm \ref{alg:restSRE-CG} either converges in $k_1 = s*\ceil{\frac{k}{s}} \leq  k + s-1$, or it converges in $k_1 \geq k+s$ iterations. The first case occurs when the norm of the residual in the  $s*\ceil{\frac{k}{s}}$ iteration remains less than $tol*||r_0||_2$. Otherwise, if the L2 norm of the residual fluctuates, then Algorithm \ref{alg:restSRE-CG} requires slightly more iterations to converge. 
 
 For the matrices {\poisso} and {\nho}, the restructured SRE-CG (Algorithm \ref{alg:restSRE-CG}) converges in exactly $k_1 = s*\ceil{\frac{k}{s}}$ iterations. For the other matrices, the three discussed cases are observed, i.e. the restructured SRE-CG  converges in $k_1$ iterations where $ k_1 \geq s*\ceil{\frac{k}{s}}$. For example, for the matrix {\skyo } with $t = 32$ and $2 \leq s \leq 10$, Algorithm \ref{alg:restSRE-CG} converges in $k_1 = s*\ceil{\frac{k}{s}}$ iterations. But for $s = 4$, Algorithm \ref{alg:restSRE-CG} converges in $k_1 = s*\ceil{\frac{k}{s}} + j$ iterations where $j=0 $ for $t=32$, $j = 4$ for $t = 8$ or $64$, $j=28$ for $t = 16$, $j=36$ for $t=4$, and $j= 52$  for $t=2$.

The s-step SRE-CG method (Algorithm \ref{alg:sstepSRE-CG1}) differs from the restructured version in the update of the approximate solutions $x_k$. As discussed in section \ref{sec:sstepS}, if there is no loss of A-orthogonality of the basis, then the s-step SRE-CG method should converge in $k_s$ iterations, where the restructured SRE-CG method converges in $k_1 = s*k_s$ iterations. This is the case for the matrices {\poisso } and {\nho} for all the tested $t$ and $s$ values ($2\leq t \leq 64$ and $2\leq s \leq 10$).

 On the other hand, for the remaining 3 matrices for some values of $s$ and $t$, the s-step SRE-CG method converges in $k_s+j$ iterations due to loss of A-orthogonality of the basis. For example, for {\skyo } matrix with $t=2,4$ and $2\leq s \leq 10$ the s-step SRE-CG method converges in exactly $k_s$ iterations. Similarly for $t=8,16, 32$ with $2\leq s \leq 8$, and for $t=64$ with $2\leq s \leq 5$. But, for $t=8,16$ with $8\leq s \leq 10$, s-step SRE-CG converges in $k_s+j$ iterations. However, for $t=32$ with $s=10$ and $t=64$ with $8\leq s \leq 10$, the s-step SRE-CG  requires more iterations to converge than the SRE-CG does for the corresponding $t$, that is why an $\times$ is placed in table \eqref{tab:SRECG}. A similar convergence behavior is observed for the matrix {\skyto}.

As expected, the Communication-Avoiding SRE-CG method (Algorithm \ref{alg:sstepSRE-CG1} with CA-Arnoldi A-orthonormalization Algorithm \ref{alg:CA-Arnoldi}) is numerically unstable  due to the enlarged monomial basis construction. Unlike the s-step version, at the $i^{th}$ iteration the $st$ vectors $AW, A^{2}W, ...,A^{s}W$ are first computed and stored in $V$, then A-orthonormalized with respect to the $(s+1)t$ previous vectors and against themselves, where $W$ is an $n \times t$ matrix containing the A-orthonormalized $A^{s(i-1) - 1}{T}(r_0)$ vectors. 
To stabilize the CA-Arnoldi A-orthonormalization (Algorithm \ref{alg:CA-Arnoldi}),  the first $t$ vectors $AW$ are A-orthonormalized with respect to the previous vectors and against themselves, and then the $(s-1)t$ vectors $A(AW), A^{2}(AW),...,A^{s-1}(AW)$ are computed, as shown in Algorithm \ref{alg:CA-Arnoldi2} . 

In Table \ref{tab:SRECG}, we test the CA SRE-CG method, where the $st$ vectors are A-orthonormalized against the previous $(s+1)t$ vectors using Algorithm \ref{alg:CA-Arnoldi2} for $k>1$. The CA SRE-CG with Algorithm \ref{alg:CA-Arnoldi2} converges at a similar rate as the s-step version for ill-conditioned matrices, such as {\skyo}, {\skyto}, and {\anio}, with $s=2, \mbox{ and }3$ only. 
 However, for the matrices {\nho} and {\poisso} CA SRE-CG converges in the same number of iterations as s-step SRE-CG, even for $s\geq 4$ (not shown in the table). This implies that CA SRE-CG should converge in approximately $\ceil{\frac{k}{s}}$ iterations for $s \geq 4$, once the ill-conditioned systems are preconditioned.
 
 \begin{algorithm}[h!]
\centering
\caption{Modified CA-Arnoldi A-orthonormalization}
{\renewcommand{\arraystretch}{1.3}
\begin{algorithmic}[1]
\Statex{\textbf{Input:} $W_{j-1}$,  $n \times t$ matrix; $k$, iteration }
\Statex{\qquad \quad $Q$,  $n \times m$ matrix, $m = (s+1)t$ in CA SRE-CG and $i = kst$ in CA SRE-CG2}
\Statex{\textbf{Output:} $V$, the $n \times st$ matrix containing the A-orthonormalized $st$ computed vectors}
\State \textbf{if}{ ($k==1$)} \textbf{then} Let $W_j = W_{j-1}$
\State \textbf{else} Let $W_j = AW_{j-1}$, A-orthonormalize $W_j$ against $Q$. 
\State \textbf{end if}
\State  A-orthonormalize $W_j$, let $V = W_j$.
\For {($i=1:s-1$)} 
\State Let $W_{j+i} = AW_{j+i-1}$ 
\State Let $V = [V \; W_{j+i}]$ 
\EndFor
\State \textbf{if}{ ($k>1$)} \textbf{then} A-orthonormalize $V$ against $Q$ \textbf{end if}
\State A-orthonormalize $V$
\end{algorithmic}}
\label{alg:CA-Arnoldi2}
\end{algorithm}

In Table \ref{tab:SRECG2}, we compare the convergence behavior of the different SRE-CG2 versions with respect to number of partitions $t$ and the $s$ values. In general, a similar behavior to the  corresponding SRE-CG versions in Table \ref{tab:SRECG} is observed.
 Yet, the SRE-CG2 versions converge faster than their corresponding SRE-CG versions and are numerically more stable. 

For $s=1$, the restructured SRE-CG2 method is equivalent to the SRE-CG2 method and converges in $k$ iterations. For $s>1$, the restructured SRE-CG2 method converges in $k_1$ iterations, where $k_1 = s*k_s \geq s*\ceil{\frac{k}{s}}$, similarly to the restructured SRE-CG method. The s-step SRE-CG2 converges in $k_s$ iterations for $s\geq 2$ for all the tested matrices. As for the communication avoiding version (Algorithm \ref{alg:sstepSRE-CG2}) with the modified CA-Arnoldi A-orthonormalization (Algorithm \ref{alg:CA-Arnoldi2}), it does not converge as fast as the s-step version for ill-conditioned matrices, such as \skyo, \skyto, and \anio, with large $s$-values ($s \geq 4$). Yet, CA SRE-CG2 converges in the same number of iterations as s-step SRE-CG2, for the matrices {\nho} and {\poisso}, even with $s\geq 4$ (not shown in the table).

\begin{table}[h!]
\setlength{\tabcolsep}{2pt}
\caption{\label{tab:SRECG2} Comparison of  convergence of different SRE-CG2 versions  (restructured SRE-CG2, s-step SRE-CG2, and CA SRE-CG2 with Algorithm \ref{alg:CA-Arnoldi2}) with respect to number of partitions $t$ and  $s$ values. }
  \centering
    \renewcommand{\arraystretch}{1.2}
  \begin{tabular}{||c||c||c||c|c|c|c|c|c|c||c|c|c|c|c|c||c|c|c||}
    \cline{4-19} 
     \cline{4-19}
    \multicolumn{1}{c}{}& \multicolumn{1}{c}{}& \multicolumn{1}{c||}{} & \multicolumn{7}{c||}{\multirow{2}{*}{\bf Restructured }} & \multicolumn{6}{c||}{\multirow{2}{*}{\bf s-step }} & \multicolumn{3}{c||}{\multirow{2}{*}{\bf CA}}\\
     \multicolumn{1}{c}{}& \multicolumn{1}{c}{}& \multicolumn{1}{c||}{}  & \multicolumn{7}{c||}{\multirow{2}{*}{\bf SRE-CG2 }} & \multicolumn{6}{c||}{\multirow{2}{*}{\bf SRE-CG2}} & \multicolumn{3}{c||}{\multirow{2}{*}{\bf SRE-CG2}}\\
        \multicolumn{1}{c}{}& \multicolumn{1}{c}{}& \multicolumn{1}{c||}{} & \multicolumn{7}{c||}{\multirow{2}{*}{}} & \multicolumn{6}{c||}{\multirow{2}{*}{}} & \multicolumn{3}{c||}{\multirow{2}{*}{}}\\
    \cline{2-19}
    \multicolumn{1}{c||}{} & \multicolumn{1}{c||}{\bf CG} & \backslashbox{$\bf t$}{$\bf s$} & \bf 1 &\bf  2  &\bf   3 &\bf  4  &\bf  5 &\bf  8 &\bf  10 &\bf  2 &\bf  3 &\bf  4  &\bf  5 &\bf  8 &\bf  10 &\bf  2 &\bf  3 & \bf 4 \\
    \hline \hline
         \multirow{6}{*}{\rotatebox[origin=c]{90}{\bf \poiss}}%%%done%%%

   &\multirow{6}{*}{195} &\bf 2   &193&	97	&65&	49&	39 &25 & 20&	97&	65&	49&	39& 25 & 20& 97&	65&	49 \\
    \cline{3-19}
   & & \bf 4    &153&	77&	51&	39&	31& 20 & 16& 77	&51&39&31& 20 &16 & 77&51&39 \\
     \cline{3-19}
    &&\bf  8    &123&	62	&41	&31&	25& 16	 &13 & 62	&41	&31	&25&	 16 & 13& 62&	41	&31   \\    
     \cline{3-19}
    &&\bf 16   &95&	48&	32	&24&	19&12	 &10 & 48&	32	&24&	19&12	 & 10& 48&	32	&24 \\
      \cline{3-19}
   & &\bf 32   &70&	36&	24&	18&	14&	9 & 8& 35&	24&	18&	14&9 &8 & 	35&	24&	18  \\
     \cline{3-19}
    &&\bf 64    &52&	26&	18&	13&	11& 7	 &6 & 26&	18&	13&	11&	7 &6 & 26&	18&	13
  \\
   \hline
  \hline
       \multirow{6}{*}{\rotatebox[origin=c]{90}{\nh}}%%%done%%%
   &\multirow{6}{*}{259} &\bf 2   & 243&	122&	81&	61&	49&31 &25 & 	122&	81&	61&	49&31 &25 & 	123&	81&	61 \\
    \cline{3-19}
   & &\bf 4 & 194&	97&	65&	49&	39& 25 & 20& 	97&	65	&49	&39&	25  &20 & 94	&65&47\\
     \cline{3-19}
    &&\bf 8    &150	&75&	50&	38&	30&	19 & 15& 75&	50&	38&	30&	19 &15 & 75&	50&	38   \\    
     \cline{3-19}
    &&\bf 16   & 113&	57&	38&	29&	23&15	 & 12& 57&	38&	29&	23&	15 & 12& 56&	38&	29 \\
      \cline{3-19}
   & &\bf 32   &  84&	42&	28&	21&	17&11 & 9& 	42&	28&	21&	17&	11 &9 & 41&	28&	21  \\
     \cline{3-19}
     
    &&\bf 64    &60&	30&	20&	15&	12&	8 &6 & 30&	20&	15&	12&8	 & 6& 30&	20&	15
  \\
   \hline
  \hline
       \multirow{6}{*}{\rotatebox[origin=c]{90}{\sky}}%%%%done%%%
   &\multirow{6}{*}{5951} &\bf 2   &1415&	708&	472&	354&	283&	177 &  142 & 708&	472&	354&	283&	177  &142 & 708&	472&	365 \\
    \cline{3-19}
   & &\bf 4 & 756&  	378&	252&	189&	152&	 95&76 & 378&	252&	189&	152&	95 &76 & 378&	252& 576\\
     \cline{3-19}
    &&\bf 8    &399&   200&	133&	100&	80& 50 &40 & 200&	133&	100&	80&	 50&40 & 199&	133&	295  \\    
     \cline{3-19}
    &&\bf 16   & 219&	110&	73&	55&	44&28  &22 & 	110&	73&	55&	44&28 &22 & 	109&	73&	147\\
      \cline{3-19}
   & &\bf 32   & 125&	63&	42&	32&	25& 16 &13 & 	63&	42&	32&	25&	16 &13 & 63&	42&	77  \\
     \cline{3-19}
     
    &&\bf 64    &74&  	37&	25&	19&	15&10 &8 & 	37&	25&	19&	15&	10 &8 & 37&	25&	39
  \\
   \hline
  \hline
       \multirow{6}{*}{\rotatebox[origin=c]{90}{\skyt}}

   &\multirow{6}{*}{902} &\bf 2   & 	570&	285&	190&	143&	114&72 &57 & 	285&	190&	144&	114	& 72& 57& 285&	190&	155\\
    \cline{3-19}
   & &\bf 4 & 375&	190&	125&	95&	75&48 & 38& 	190&	125&	95&	75&	48 & 38& 190	&127&	101\\
     \cline{3-19}
    &&\bf 8    &213&	107&	71&	54&	43&27	 &22 & 107&	71&	54&	43&27 & 22& 	107&	71&	224  \\    
     \cline{3-19}
    &&\bf 16   & 117&	59&	39&	30&	24&15 & 12& 	59&	39&	30&	24& 15 &12 &	59&	39&	124\\
      \cline{3-19}
   & &\bf 32   &  69&	35&	23&	18&	14&9 & 7& 	35&	23&	18&	14&9 & 7& 	35&	23&	69   \\
     \cline{3-19}    
    &&\bf 64    &43&	22&	15&	11&	9&6 &5 & 	22&	15&	11&	9& 6& 5& 	22&	15&	x  \\
   \hline
  \hline  
       \multirow{6}{*}{\rotatebox[origin=c]{90}{\ani}}%%%%done%%%
   &\multirow{6}{*}{4146} &\bf 2   & 	875&	438&	292&	219&	175&	110 &88 & 438&	292&	219&	175&	110 &88 & 438&	292&	219 \\
    \cline{3-19}
   & &\bf 4 &673&	340&	229&	170&	131&81 &66 & 	340&	229&	170&	131&81	 &66 & 340&	229&	170\\
     \cline{3-19}
    &&\bf 8    &449&	225&	150&	113&	90&	57 &45 & 225&	150&	113&	90&57 &45 & 	225&	150&	113 \\    
     \cline{3-19}
    &&\bf 16   & 253&	127&	85&	64&	51& 32& 26& 	127&	85&	64&	51&32	 & 26& 127&	85&	78\\
      \cline{3-19}
   & &\bf 32   & 148&	74&	49&	37&	30&19 &15 & 	74&	50&	37&	30&19 & 15& 	74&	50&	58   \\
     \cline{3-19}    
    &&\bf 64    &92&	46&	31&	23&	19&12 & 10& 	46&	31&	23&	19&12	 & 10& 46&	31&	31\\
   \hline
  \hline
  \end{tabular}
\end{table}
 
\begin{table}[h!]
\setlength{\tabcolsep}{2pt}
\caption{\label{tab:MSDOCG} Comparison of the convergence of different MSDO-CG versions  (s-step MSDO-CG, CA MSDO-CG with Algorithm \ref{alg:CA-Arnoldi}, and CA MSDO-CG with Algorithm \ref{alg:CA-Arnoldi2}) with respect to number of partitions $t$ and  $s$ values. }
  \centering
  %\tabsize    
    \renewcommand{\arraystretch}{1.2}
  \begin{tabular}{||c||c||c||c||c|c|c|c|c|c|c|c|c|||c|c|c|c|||c|c|c|c|c|c||}
    \cline{4-23} 
     \cline{4-23}
    \multicolumn{1}{c}{}& \multicolumn{1}{c}{}& \multicolumn{1}{c||}{} &\multicolumn{1}{c||}{\multirow{2}{*}{\bf MSD }} & \multicolumn{9}{c||}{\multirow{2}{*}{\bf s-step }} & \multicolumn{4}{c||}{\multirow{2}{*}{\bf CA MSDO-CG}}& \multicolumn{6}{c||}{\multirow{2}{*}{\bf CA MSDO-CG}}\\
     \multicolumn{1}{c}{}& \multicolumn{1}{c}{}& \multicolumn{1}{c||}{} & \multicolumn{1}{c||}{\multirow{2}{*}{\bf OCG }}  & \multicolumn{9}{c||}{\multirow{2}{*}{\bf MSDO-CG }} & \multicolumn{4}{c||}{\multirow{2}{*}{\bf  with Algorithm\ref{alg:CA-Arnoldi}}}& \multicolumn{6}{c||}{\multirow{2}{*}{\bf with Algorithm\ref{alg:CA-Arnoldi2}}}\\
        \multicolumn{1}{c}{}& \multicolumn{1}{c}{}& \multicolumn{1}{c||}{} & \multicolumn{1}{c||}{} & \multicolumn{9}{c||}{\multirow{2}{*}{}} & \multicolumn{4}{c||}{\multirow{2}{*}{}} & \multicolumn{6}{c||}{\multirow{2}{*}{}} \\
    \cline{2-23}
    \multicolumn{1}{c||}{} & \multicolumn{1}{c||}{\bf CG} & \backslashbox{$\bf t$}{$\bf s$} & \bf 1 &  \bf 1 &\bf  2  &\bf   3 &\bf  4  &\bf  5 &  \bf 6 &  \bf 7 &\bf  8 &\bf  10 &\bf  2 &\bf  3 &\bf  4  &\bf  5 &\bf  2 &\bf  3 &\bf  4  &\bf  5 &\bf  6 &\bf  7    \\
    \hline \hline
         \multirow{6}{*}{\rotatebox[origin=c]{90}{\bf \poiss}}
   &\multirow{6}{*}{195} &\bf 2   &198 &198 & 99	&66&	49&	40 & 33&	28&	23&	18 & 	
   99&	66&	50&	40 &99& 66 & 49& 34&	33&	28 \\
    \cline{3-23}
   & & \bf 4    &166 &166 &	83&	56&	42& 34 & 28& 24	&21&17&  83&	56&	42&	34& 83& 55 &42 & 33&28&25 \\
     \cline{3-23}
    &&\bf  8    &137 &137 &	68	&46	&34&	27& 23	 &19 & 17	&13   	& 69&	46&	36&	28&
68&	 46 & 35& 28&	24	&21   \\    
     \cline{3-23}
    &&\bf 16   &121 &121 &	59&	39	&29&	23&18&16 & 14& 11 
    &61&	41&	31&	25&
	59&38	 & 28&23&19	&16 \\
      \cline{3-23}
   & &\bf 32   &95 &95 &	45&	29&	22&	17&	14 & 12& 10&	8& 48&	32&	25&	20&
   	45&30 &22 & 	18&	15&	13  \\
     \cline{3-23}
    &&\bf 64    &69 &69 & 33&	21&	16&	12& 10	 &9 & 8&	6& 37&	25&	20&	16&
    	33&	22 &16 & 14&	12&	10
  \\
   \hline
  \hline
         \multirow{6}{*}{\rotatebox[origin=c]{90}{\nh}}

   &\multirow{6}{*}{259} &\bf 2   & 255&	255	&127&84&	63&	51&	42&	36&	32&	26&	127&	85&	64&	51& 127&	84&	63&	51&	42& 	37 \\
    \cline{3-23}
   & &\bf 4 & 210&	210	&104	&69	&52&	42&	34&	30&	26&	20&	105&	71&	53&	42&
104&	68&	52&	42&	35&	31\\
     \cline{3-23}
    &&\bf 8    &170&	170	&84	&56	&42&	33&	27&	23&	20&	16&85&	57&	43&	34&
	84&	56&	42&	33&	29&	25   \\    
     \cline{3-23}
    &&\bf 16   & 138&	138	&68	&44	&33&	26&	21&	18&	16&	12&70&	47&	36&	28&
	68&	44&	33&	27&	23&	20 \\
      \cline{3-23}
   & &\bf 32   & 106	&106&51	&33&	25&	19&	16&	13&	12&	10&	54&	36&	28&	23&
51&	33&	25&	21&	18&	16\  \\
     \cline{3-23}
     
    &&\bf 64    &76	&76&	37	&24&	17&	14&	11&	10&	9&	7&41&	28&	22&	17&
	37&	24&	19&	16&	14&	12  \\
   \hline
  \hline
       \multirow{6}{*}{\rotatebox[origin=c]{90}{\sky}}

   &\multirow{6}{*}{5951} &\bf 2   &1539&1539&719&480&360&288&240&206&	180&	144&778&	527&	401&	327& 	720&493&374&307&259&224 \\
    \cline{3-23}
   & &\bf 4 &916&	916&	397&	259&	194&	154&	129&	110&	96&	77&466&	312&x&x&	395&	271&	207&	170&	143&	x\\
     \cline{3-23}
    &&\bf 8    &517&	517&	214&	141&	105&	84&	70&	60&	52&	42&260&	167&	129&	x&	214&	149&	114&	95&	80&	x \\    
     \cline{3-23}
    &&\bf 16   & 277&	277&	122&	81&	60&	47&	40&	34&	30&	24&141&	95&	x&	x&	122&	85&	66&	54&	x&	x\\
      \cline{3-23}
   & &\bf 32   & 192&	192&	74&	48&	36&	28&	23&	20&	17&	14&92&	61&	x&	x&	73&	51&	40&	33&	x&	x  \\
     \cline{3-23}
     
    &&\bf 64    &123&	123&	47&	29&	22&	17&	14&	12&	11&	8&60&	x&	x&	x&	47&	32&	25&	23&	x&	x  \\
   \hline
  \hline
       \multirow{6}{*}{\rotatebox[origin=c]{90}{\skyt}}
       
   &\multirow{6}{*}{902} &\bf 2   & 	637&	633&	334&	211&	169&	139&	111&	89&	81&	66&	339&	235&	180&	153&333&	242&	193&	160&	134&	119\\
    \cline{3-23}
   & &\bf 4 & 374&	373&	205&	137&	103&	81&	68&	58&	51&	41&	204&	138&	107&	86&
206&	140&	109&	89&	75&	66\\
     \cline{3-23}
    &&\bf 8    &224&	224&	112&	74&	56&	44&	37&	32&	28&	22&117&	82&	63&	50&
	112&	78&	61&	49&	42&	37 \\    
     \cline{3-23}
    &&\bf 16   &137&	137&	63&	42&	31&	25&	21&	18&	16&	13&	73&	50&	38&	32&
63&	45&	35&	29&	25&	22\\
      \cline{3-23}
   & &\bf 32   &  89&	89&	38&	25&	19&	15&	13&	11&	9&  8& 45&	30&	23&	x&
	38&	27&	21&	18&	19&	x\\
     \cline{3-23}    
    &&\bf 64    &50&	50&	24&	16&	12&	10&	8&	7&	6&	5&27&	19&	16&	x&
	24&	17&	14&	12&	x&	x \\
   \hline
  \hline  
       \multirow{6}{*}{\rotatebox[origin=c]{90}{\ani}}

   &\multirow{6}{*}{4146} &\bf 2   & 	896&	896&	456&	301&	227&	181&	152&	130&	114&	91&458&	312&	237&	195&	452&	309&	242&	201&	168&	148 \\
    \cline{3-23}
   & &\bf 4 &796&	796&	362&	238&	177&	140&	115&	99&	88&	69&413&	281&	216&	176&	362&	253&	194&	159&	135&	117\\
     \cline{3-23}
    &&\bf 8    &473&	473&	231&	154&	115&	92&	77&	66&	58&	46&	238&	163&	126&	x&231&	158&	120&	98&	83&	x \\    
     \cline{3-23}
    &&\bf 16   & 292&	292&	130&	86&	64&	51&	43&	37&	32&	26&140&	95&	x&	x&	130&	91&	70&	57&	55&	x\\
      \cline{3-23}
   & &\bf 32   &213&	213&	77&	51&	38&	30&	25&	22&	19&	15&97&	61&	x&	x&	76&	55&	42&	35&	x&	x \\
     \cline{3-23}    
    &&\bf 64    &115&	115&	48&	31&	24&	19&	16&	14&	12&	10&	56&	x&	x&	x&48&	34&	27&	24&	x&	x\\
   \hline
  \hline
  \end{tabular}
\end{table}  
 In Table \ref{tab:MSDOCG}, we compare the convergence behavior of MSDO-CG, s-step MSDO-CG,  CA MSDO-CG with Algorithm \ref{alg:CA-Arnoldi}, and CA MSDO-CG with Algorithm \ref{alg:CA-Arnoldi2} (where $W_{j-1} = [T(r_{k-1})]$ and $W_j = W_{j-1}$ for $k \geq 1$) versions with respect to number of partitions $t$ and the $s$ values. We do not test a restructured MSDO-CG since the s-step version is not exactly equivalent to the merging of $s$ iterations of MSDO-CG. The s-step MSDO-CG with $s=1$ is equivalent to a modified version of MSDO-CG which differs algorithmically from MSDO-CG but is equivalent theoretically. Moreover, MSDO-CG and s-step MSDO-CG with $s=1$ converge in the same number of iterations for all $t$ values and matrices. For $s\geq 2$, s-step MSDO-CG converges in $m$ iterations where in most cases $m \leq \ceil{\frac{k}{s}}$ and MSDO-CG converges in $k$ iterations. Moreover, for all the matrices, the s-step MSDO-CG converges for $s = 10$ and all values of $t$.

 Unlike the CA SRE-CG and CA SRE-CG2 with the CA-Arnoldi Algorithm \ref{alg:CA-Arnoldi}, the CA MSDO-CG with Algorithm \ref{alg:CA-Arnoldi} converges for $s= 2$, and $3$, as shown in table \ref{tab:MSDOCG}. The difference is that  in the SRE-CG and SRE-CG2 we are computing a modified block version of the powers method, where $t$ vectors ($T(r_0)$) are multiplied by powers of A and are A-orthonormalized. Thus there is a higher chance that these vectors converge to the largest eigenvector in a very fast rate, leading to a numerically linearly dependent basis. Whereas, in CA MSDO-CG at every iteration we are computing a block version of the powers method but starting with a new set of $t$ vectors, i.e. $T(r_{k-1})$ at the $k^{th}$ iteration. For the matrices {\nho} and {\poisso}, CA MSDO-CG scales even for $s>5$. But for the other matices, as $s$ grows, the CA MSDO-CG requires much more than $\ceil{\frac{k}{s}}$ and $\ceil{\frac{k}{s-1}}$ iterations to converge, due to the stagnation of the relative error.

For the matrices {\nho} and {\poisso}, CA MSDO-CG with Algorithm \ref{alg:CA-Arnoldi2} converges in exactly the same number of iterations as CA MSDO-CG with Algorithm \ref{alg:CA-Arnoldi} and s-step MSDO-CG up to $s=10$.  %Otherwise, a similar behavior is observed.
On the other hand, the CA MSDO-CG with Algorithm \ref{alg:CA-Arnoldi2} converges faster than CA MSDO-CG with Algorithm \ref{alg:CA-Arnoldi} version for the corresponding $s$ and $t$ values, for the matrices {\skyto} (except for $t = 2,4$), {\skyo}, and {\anio} (except for $t = 2,4$). 
 More importantly, CA MSDO-CG with Algorithm \ref{alg:CA-Arnoldi2} is numerically more stable than CA MSDO-CG with Algorithm \ref{alg:CA-Arnoldi} and CA SRE-CG2 with Algorithm \ref{alg:CA-Arnoldi2}, as it scales up to at least $s=5$, or $6$. Whereas CA SRE-CG2 with Algorithm \ref{alg:CA-Arnoldi2} and CA MSDO-CG with Algorithm \ref{alg:CA-Arnoldi} scales up to $s=3$, or $4$ as shown in Tables \ref{tab:SRECG2} and \ref{tab:MSDOCG}.

As a summary, for the well-conditioned matrices such as {\nho} and {\poisso},the s-step and communication avoiding with Algorithm \ref{alg:CA-Arnoldi2} versions  of SRE-CG and SRE-CG2 converge in the same number of iterations and scale up to at least $s=10$. But the communication avoiding with Algorithm \ref{alg:CA-Arnoldi} version  of SRE-CG and SRE-CG2 do not converge due to the instability in the basis construction, specifically the A-orthonormalization process.
On the other hand, the s-step, communication avoiding with Algorithm \ref{alg:CA-Arnoldi} and communication avoiding with Algorithm \ref{alg:CA-Arnoldi2} versions of MSDO-CG for the matrices {\nho} and {\poisso} converge  in the same number of iterations and scale up to at least $s=10$. Moreover, the corresponding versions of SRE-CG, SRE-CG2, and MSDO-CG converge in approximately the same number of iterations.

For the other matrices, the s-step versions of SRE-CG, SRE-CG2, and MSDO-CG converge and scale up to at least $s=10$, as expected. The communication avoiding with Algorithm \ref{alg:CA-Arnoldi} versions  of SRE-CG and SRE-CG2 do not converge. But the communication avoiding MSDO-CG with Algorithm \ref{alg:CA-Arnoldi} converges. Moreover,  the communication avoiding MSDO-CG with Algorithm \ref{alg:CA-Arnoldi2} scales better than  the communication avoiding SRE-CG2 with Algorithm \ref{alg:CA-Arnoldi2}, even though it might require more iterations. 

\section{The Preconditioned Versions} \label{sec:precCG}
Krylov subspace methods are rarely used without preconditioning. Moreover, Conjugate Gradient is a method for solving symmetric positive definite matrices. For this purpose, split preconditioned versions of the above-mentioned s-step methods for solving the system $L^{-1}AL^{-t}(L^{t}x) = L^{-1}b$ are introduced, where the preconditioner is $M = LL^t$. Then, the numerical stability of the preconditioned methods is briefly discussed.

\subsection{Preconditioned Algorithms}
One possible way for preconditioning the s-step versions is by simply replacing $A$ by $L^{-1}AL^{-t}$ and $b$ by $L^{-1}b$ in the algorithms, where $L^{-1}AL^{-t}y = L^{-1}b$ is first solved and then the solution $x$ is obtained by solving $y = L^{t}x$. In \cite{sophiethesis}, MSDO-CG is preconditioned in this manner (Algorithm 40), where the vectors are $L^{-1}AL^{-t}$-orthonormalized (Algorithms 19 and 22) rather than $A$-orthonormalized. In this paper we will precondition the s-step and communication avoiding methods by avoiding the use of $L^{-1}AL^{-t}$-orthonormalization. 

Given the following system $\widehat{A}\widehat{x} = \widehat{b}$, where $\widehat{A} = L^{-1}AL^{-t}$, $\widehat{x} = L^{t}x$, and $\widehat{b} = L^{-1}b$. The following relations summarized the SRE-CG, SRE-CG2, and modified MSDO-CG methods for this system:
\begin{eqnarray}
 \widehat{\alpha}_k &=& \widehat{V}^t_k \widehat{r}_{k-1} \nonumber \\
 \widehat{x}_k &=& \widehat{x}_{k-1} + \widehat{V}_k\widehat{\alpha}_k \nonumber\\
 \widehat{r}_k &=& \widehat{r}_{k-1} - \widehat{A}\widehat{V}_k\widehat{\alpha}_k \nonumber
\end{eqnarray}

 The difference is in how the $\widehat{V}_k$ vectors are constructed. In the modified MSDO-CG, $\widehat{V}_k$ is set to  $[T(\widehat{r}_{k-1})]$, and then $\widehat{A}$-orthonormalized against all previous vectors.  In SRE-CG and SRE-CG2 methods,
 \begin{eqnarray}\widehat{V}_k = \begin{cases}
 [T(\widehat{r}_0)], & \mbox{ if $k = 1$ }\\
 \widehat{A}\widehat{V}_{k-1}, & \mbox{ if $k \geq 2$ }
 \end{cases} \nonumber \end{eqnarray}
  and then $\widehat{V}_k$ is $\widehat{A}$-orthonormalized against the previous $2t$ vectors (SRE-CG) or against all previous vectors (SRE-CG2). In the three methods, $\widehat{V}_i^t \widehat{A} \widehat{V}_i = I$ and $\widehat{V}_k^t \widehat{A} \widehat{V}_i = 0$ where $i = k-2,k-1$ for SRE-CG and $i < k$ for SRE-CG2 and modified MSDO-CG .

 Note that $\widehat{r}_{k} = \widehat{b} - \widehat{A}\widehat{x}_{k} = L^{-1}b - L^{-1}AL^{-t}L^{t}x_{k} =  L^{-1}(b - Ax_{k}) = L^{-1}r_{k}$. Thus, we derive the corresponding equations for $x_k$, and $r_k$. \vspace{2mm}
 \begin{eqnarray}
 \widehat{\alpha}_k &=& \widehat{V}^t_k \widehat{r}_{k-1} = \widehat{V}^t_k L^{-1}r_{k} = (L^{-t}\widehat{V}_k)^tr_k \nonumber \\
 \widehat{x}_k &=& L^{t}x_k = \widehat{x}_{k-1} + \widehat{V}_k\widehat{\alpha}_k = L^{t}{x}_{k-1} + \widehat{V}_k\widehat{\alpha}_k \;\;\;
 \implies x_k = {x}_{k-1} + (L^{-t}\widehat{V}_k)\widehat{\alpha}_k \nonumber \\
  \widehat{r}_k &=& L^{-1}r_{k} = \widehat{r}_{k-1} - \widehat{A}\widehat{V}_k\widehat{\alpha}_k = L^{-1}{r}_{k-1} - L^{-1}AL^{-t}\widehat{V}_k\widehat{\alpha}_k \nonumber \\
 \implies r_{k} &=& {r}_{k-1} - A(L^{-t}\widehat{V}_k)\widehat{\alpha}_k \nonumber
\end{eqnarray}

Let $V_k = L^{-t}\widehat{V}_k$, then 
\begin{eqnarray}
\widehat{\alpha}_k &=& V_k^tr_k,\nonumber \\
 x_k &=& {x}_{k-1} + {V}_k\widehat{\alpha}_k,\nonumber\\
 r_{k} &=& {r}_{k-1} - A{V}_k\widehat{\alpha}_k.\nonumber
 \end{eqnarray}
 
  Moreover, $T(\widehat{r}_{k}) = T(L^{-1}r_{k})$ and $\widehat{A}\widehat{V}_{k-1} =  L^{-1}AL^{-t}\widehat{V}_{k-1} =  L^{-1}A{V}_{k-1}$. As for the $\widehat{A}$-orthonormalization, we require that  $\widehat{V}_k^t \widehat{A} \widehat{V}_i = 0$ for some values of $i\neq k$. But $$\widehat{V}_k^t \widehat{A} \widehat{V}_i = \widehat{V}_k^t L^{-1}AL^{-t} \widehat{V}_i = (L^{-t}\widehat{V}_k)^tA(L^{-t} \widehat{V}_i) = {V}_k^tA{V}_i.$$ Thus, it is sufficient to A-orthonormalize $V_k = L^{-t}\widehat{V}_k$ instead of  $\widehat{A}$-orthonormalizing $\widehat{V}_k$, where in modified MSDO-CG, $$V_k = L^{-t}[T(\widehat{r}_{k-1})] =  L^{-t}[T(L^{-1}r_{k})],$$ and in SRE-CG and SRE-CG2 $$V_k = L^{-t}\widehat{V}_k = \begin{cases}
 L^{-t}[T(\widehat{r}_0)], & \mbox{ if } k = 1\\
 L^{-t} L^{-1}{A}{V}_{k-1} = M^{-1}A{V}_{k-1}, & \mbox{ if $k \geq 2$ }
 \end{cases}.$$
This summarizes the three methods for $s=1$. In general, for $s>1$ the s-step methods are described in Algorithms \ref{alg:psstepSRE-CG1}, \ref{alg:psstepSRE-CG2}, and \ref{alg:psstepMSDO-CG}. 
As for the communication avoiding versions, in Algorithms \ref{alg:CA-Arnoldi} and \ref{alg:CA-Arnoldi2}, $AW_{j+i-1}$ is replaced by $M^{-1}AW_{j+i-1}$, and $W_{j-1} = L^{-t}[T(L^{-1}r_k)]$, for $k\geq 1$ in CA MSDO-CG and for $k= 1$ in CA SRE-CG and CA SRE-CG2.

If the preconditioner is a block diagonal preconditioner, with $t$ blocks that correspond to the $t$ partitions of the matrix $A$, then $[T(L^{-1}r_k)] = L^{-1}[T(r_k)]$ and $L^{-t}[{T}(L^{-1}r_k)] = M^{-1}[T(r_k)]$. In this case, no need for split preconditioning, similarly to CG. \vspace{-5mm}
 \begin{algorithm}[H]
\centering
\caption{Split preconditioned s-step SRE-CG }
{\renewcommand{\arraystretch}{1.3}
\begin{algorithmic}[1]
\Statex{\textbf{Input:} $A$,  $n \times n$ symmetric positive definite matrix; $k_{max}$, maximum allowed iterations}
\Statex{\qquad \quad $b$,  $n \times 1$ right-hand side; $x_0$, initial guess; $\epsilon$, stopping tolerance; $M = LL^t$; $s$ }
\Statex{\textbf{Output:} $x_k$, approximate solution of the system $L^{-t}AL^t(L^{-t}x)=L^{-t}b$}
\State$r_0 = b - Ax_0$, $\rho_0 = ||r_0||_2$ , $\rho =\rho_0$, $\widehat{r}_0 = L^{-1}r_{0}$, $k = 1$; 
%\State 
\While {( ${\rho} > \epsilon \rho_0$ and $k < k_{max}$ )}
\State Let $j = (k-1)s+1$
\If {($k==1$)}
\State A-orthonormalize $W_j = L^{-t}[{T}(\widehat{r}_0)]$,  and let $V = W_j$ 
\Else 
\State A-orthonormalize $W_{j} = M^{-1}AW_{j-1}$ against $W_{j-2}$ and $W_{j-1}$
\State A-orthonormalize $W_{j}$ and let $V = W_{j}$  
\EndIf
\For {($i=1:s-1$)} 
\State A-orthonormalize $W_{j+i} = M^{-1}AW_{j+i-1}$ against $W_{j+i-2}$ and $W_{j+i-1}$
\State A-orthonormalize $W_{j+i}$ and let $V = [V \; W_{j+i}]$ 
\EndFor
\State $\widehat{\alpha} = V^t r_{k-1}$, \;\;\; $x_k = x_{k-1} + V \widehat{\alpha} $ 
\State $r_k = r_{k-1} - AV \widehat{\alpha} $,  \;\;\;  $\rho = ||r_{k}||_2$,  \;\;\; $k = k+1$  
\EndWhile
\end{algorithmic}}
\label{alg:psstepSRE-CG1}
\end{algorithm}\vspace{-15mm}
\begin{algorithm}[H]
\centering
\caption{Split preconditioned s-step SRE-CG2  }
{\renewcommand{\arraystretch}{1.3}
\begin{algorithmic}[1]
\Statex{\textbf{Input:} $A$,  $n \times n$ symmetric positive definite matrix; $k_{max}$, maximum allowed iterations}
\Statex{\qquad \quad $b$,  $n \times 1$ right-hand side; $x_0$, initial guess; $\epsilon$, stopping tolerance; $M = LL^t$; $s$ }
\Statex{\textbf{Output:} $x_k$, approximate solution of the system $L^{-t}AL^t(L^{-t}x)=L^{-t}b$}
\State$r_0 = b - Ax_0$, $\rho_0 = ||r_0||_2$ , $\rho = \rho_0 $, $\widehat{r}_0 = L^{-1}r_{0}$, $k = 1$; 
%\State 
\While {( ${\rho} > \epsilon \rho_0 $ and $k < k_{max}$ )}
\State Let $j = (k-1)s+1$
\If {($k==1$)}
\State A-orthonormalize $W_j = L^{-t}[{T}(\widehat{r}_0)]$, let $Q = W_j$, and $V = W_j$ 
\Else 
\State A-orthonormalize $W_j = M^{-1}AW_{j-1}$ against $Q$
\State A-orthonormalize $W_j$,  let $Q = [Q, \; W_j]$, and $V = W_j$
\EndIf
%\State Let   
\For {($i=1:s-1$)} 
\State A-orthonormalize $W_{j+i} = M^{-1}AW_{j+i-1}$ against $Q$
\State A-orthonormalize $W_{j+i}$,  let  $V = [V, \; W_{j+i}]$ and $Q = [Q, \; W_{j+i}]$ 
\EndFor
\State $\widehat{\alpha} = V^t r_{k-1}$,  \;\;\; $x_k = x_{k-1} + V \widehat{\alpha} $ 
\State $r_k = r_{k-1} - AV \widehat{\alpha} $,  \;\;\;  $\rho = ||r_{k}||_2$,  \;\;\;  $k = k+1$  
\EndWhile
\end{algorithmic}}
\label{alg:psstepSRE-CG2}
\end{algorithm}
\begin{algorithm}[H]
\centering
\caption{Split preconditioned s-step MSDO-CG  }
{\renewcommand{\arraystretch}{1.3}
\begin{algorithmic}[1]
\Statex{\textbf{Input:} $A$,  $n \times n$ symmetric positive definite matrix; $k_{max}$, maximum allowed iterations}
\Statex{\qquad \quad $b$,  $n \times 1$ right-hand side; $x_0$, initial guess; $\epsilon$, stopping tolerance; $M = LL^t$; $s$ }
\Statex{\textbf{Output:} $x_k$, approximate solution of the system $L^{-t}AL^t(L^{-t}x)=L^{-t}b$}
\State$r_0 = b - Ax_0$, $\rho_0 = ||r_0||_2$, $\rho =\rho_0$, $k = 1$; 
\While {( ${\rho} > \epsilon \rho_0 $ and $k < k_{max}$ )}
\State$\widehat{r}_{k-1} = L^{-1}r_{k-1}$ and  $W_1 = L^{-t}[{T}(\widehat{r}_{k-1})]$
\If {($k==1$)}
\State A-orthonormalize $W_1$, let $V = W_1$ and $Q = W_1$ 
\Else 
\State A-orthonormalize $W_1$ against $Q$
\State A-orthonormalize $W_1$,  let $V = W_1$ and $Q = [Q \; W_1]$ 
\EndIf
\For {($i=1:s-1$)} 
\State A-orthonormalize $W_{i+1} = M^{-1}AW_i$ against $Q$
\State A-orthonormalize $W_{i+1}$,  let $V = [V \; W_{i+1}]$ and $Q = [Q \; W_{i+1}]$ 
\EndFor
\State $\widehat{\alpha} = V^t r_{k-1}$,  \;\;\;  $x_k = x_{k-1} + V \widehat{\alpha} $ 
\State $r_k = r_{k-1} - AV\widehat{\alpha} $,  \;\;\; $\rho = ||r_{k}||_2$,  \;\;\;  $k = k+1$  
\EndWhile
\end{algorithmic}}
\label{alg:psstepMSDO-CG}
\end{algorithm}

\subsection{Convergence}\label{sec:precconv}
We test the preconditioned versions using block Jacobi preconditioner. First, the graphs of the matrices are partitioned into 64 domains using Metis Kway dissection \cite{metis}. Each of the 64 diagonal blocks is factorized using Cholesky decomposition (Table \ref{tab:precECG}) or Incomplete Cholesky zero fill-in decomposition (Table \ref{tab:precECG2}). 

Then for a given $t$, each of the $t$ domains is the union of $64/t$ consecutive domains, where the preconditioner $M = LL^t$, the $L_i$'s are lower triangular blocks for $i=1,2,..,64$, and
$$L =  \begin{bmatrix}%{cccccc}
L_1 &0&0&0&\hdots&0\\
0&L_2&0&0&\hdots&0\\
0&0&\ddots&0&\hdots&0\\
0&0&0&L_i&\hdots&0\\
0&0&0&0&\ddots&0\\
0&0&0&\hdots&0&L_{64}
\end{bmatrix}. $$

In Table \ref{tab:precECG2}, we test the convergence of Incomplete Cholesky block Jacobi preconditioned s-step and CA versions of SRE-CG, SRE-CG2 and MSDO-CG for $t = 2,4,8,16,32,64$ and $s = 1,2,4,8$. For the matrices {\poisso}, {\nho}, {\skyo}, and {\anio}, and for all the $s$ and $t$ values, the preconditioned s-step versions and their corresponding CA versions with Algorithm \ref{alg:CA-Arnoldi2} converge in the same number of iterations and scale for $s\geq 8$. 
 CA SRE-CG with Algorithm \ref{alg:CA-Arnoldi} stagnates, whereas CA SRE-CG2 with Algorithm \ref{alg:CA-Arnoldi} converges in exactly the same number of iterations as s-step SRE-CG2 and CA SRE-CG2 with Algorithm \ref{alg:CA-Arnoldi2}. Moreover, the corresponding preconditioned SRE-CG and SRE-CG2 versions converge in the similar number of iterations.  
As for {\skyto}, the CA SRE-CG stagnates for $s=8$ and $t=4,8$ only. The CA MSDO-CG with Algorithm \ref{alg:CA-Arnoldi2} converges as fast as s-step MSDO-CG, whereas CA MSDO-CG with Algorithm \ref{alg:CA-Arnoldi} requires more iterations, in some cases ( {\skyo}, {\skyto})

 A similar convergence behavior is observed for the Complete Cholesky block Jacobi preconditioned s-step and CA versions of SRE-CG, SRE-CG2 and MSDO-CG, in Table \ref{tab:precECG}, where the only difference is that the methods converge faster than the corresponding Incomplete Cholesky block Jacobi preconditioned versions.

\begin{table}[H]
\setlength{\tabcolsep}{2pt}
\caption{\label{tab:precECG2} Comparison of the convergence of different Block Jacobi with Incomplete Cholesky preconditioned E-CG versions  (s-step SRE-CG, CA SRE-CG with Algorithm \ref{alg:CA-Arnoldi2}, s-step SRE-CG2, CA SRE-CG2 with Algorithm \ref{alg:CA-Arnoldi} or \ref{alg:CA-Arnoldi2}, s-step MSDO-CG, CA MSDO-CG with Algorithm \ref{alg:CA-Arnoldi} and Algorithm \ref{alg:CA-Arnoldi2}) with respect to number of partitions $t$ and $s$ values.}
  \centering
    \renewcommand{\arraystretch}{1.2}
  \begin{tabular}{||c||c||c||c|c|c|c||c|c|c|||c|c|c|c||c|c|c|||c|c|c|c||c|c|c||c|c|c||}
    \cline{4-27} 
     \cline{4-27}
     \multicolumn{1}{c}{}& \multicolumn{1}{c}{}& \multicolumn{1}{c||}{} & \multicolumn{7}{c|||}{\multirow{1}{*}{\bf SRE-CG }} & \multicolumn{7}{c|||}{\multirow{1}{*}{\bf SRE-CG2}}&  \multicolumn{10}{c||}{\multirow{1}{*}{\bf  MSDO-CG}}\\     \cline{4-27} 
    \multicolumn{1}{c}{}& \multicolumn{1}{c}{}& \multicolumn{1}{c||}{} & \multicolumn{4}{c||}{\multirow{2}{*}{\bf s-step  }} & \multicolumn{3}{c|||}{\multirow{2}{*}{\bf CA Alg7}}& \multicolumn{4}{c||}{\multirow{2}{*}{\bf s-step }}& \multicolumn{3}{c|||}{\multirow{2}{*}{\bf CA Alg5/7}}& \multicolumn{4}{c||}{\multirow{2}{*}{\bf s-step }}& \multicolumn{3}{c||}{\multirow{2}{*}{\bf CA Alg5}}& \multicolumn{3}{c||}{\multirow{2}{*}{\bf CA Alg7}}\\
        \multicolumn{1}{c}{}& \multicolumn{1}{c}{}& \multicolumn{1}{c||}{}& \multicolumn{4}{c||}{\multirow{2}{*}{}} & \multicolumn{3}{c|||}{\multirow{2}{*}{}} & \multicolumn{4}{c||}{\multirow{2}{*}{}} & \multicolumn{3}{c|||}{\multirow{2}{*}{}} & \multicolumn{4}{c||}{\multirow{2}{*}{}} & \multicolumn{3}{c||}{\multirow{2}{*}{}}& \multicolumn{3}{c||}{\multirow{2}{*}{}}\\
    \cline{2-27}
    \multicolumn{1}{c||}{} & \multicolumn{1}{c||}{\bf PCG} & \backslashbox{$\bf t$}{$\bf s$} & \bf 1 &\bf  2  &\bf   4 &\bf  8  &\bf  2  &\bf   4 &\bf  8 &\bf  1 &\bf  2  &\bf  4  &\bf  8 &\bf  2 &\bf  4 &\bf 8  &\bf  1 &\bf  2  &\bf   4 &\bf  8 &\bf  2  &\bf   4 &\bf  8 &\bf  2  &\bf   4 &\bf  8 \\
    \hline \hline        
             \multirow{6}{*}{\rotatebox[origin=c]{90}{\bf \poiss}}
   &\multirow{6}{*}{86} &\bf 2   &82&	41&	21&	11&	41&	21&	11&	82&	41&	21&	11&	41&	21&	11&	79&	40&	21&	11&	40&	20&	11&	40&	21&	12  \\
    \cline{3-27}
   & & \bf 4    &65&	33&	17&	9&	33&	17&	9&	65&	33&	17&	9&	33&	17&	9&	71&	36&	18&	9&	35&	18&	9&	36&	18&	10 \\
     \cline{3-27}
    &&\bf  8    &55&	28&	14&	7&	28&	14&	7&	55&	28&	14&	7&	28&	14&	7&	59&	30&	14&	7&	32&	16&	8&	30&	15&	9   \\    
     \cline{3-27}
    &&\bf 16   &41&	21&	11&	6&	21&	11&	6&	41&	21&	11&	6&	21&	11&	6&	51&	26&	12&	6&	27&	14&	7&	26&	12&	7 \\
      \cline{3-27}
   & &\bf 32   &30&	15&	8&	4&	15&	8&	4&	30&	15&	8&	4&	15&	8&	4&	40&	21&	9&	5&	23&	12&	5&	21&	10&	6 \\
     \cline{3-27}
    &&\bf 64    &23&	12&	6&	3&	12&	6&	3&	23&	12&	6&	3&	12&	6&	3&	30&	16&	7&	3&	19&	9&	4&	16&	8&	5  \\
   \hline
  \hline
         \multirow{6}{*}{\rotatebox[origin=c]{90}{\nh}}
   &\multirow{6}{*}{115} &\bf 2   & 105&	53&	27&	14&	53&	27&	14&	105&	53&	27&	14&	53&	27&	14&	106&	53&	27&	13&	54&	27&	14&	53&	27&	15 \\
    \cline{3-27}
   & &\bf 4 & 82&	41&	21&	11&	41&	21&	11&	82&	41&	21&	11&	41&	21&	11&	87&	45&	22&	11&	46&	22&	12&	45&	22&	13 \\
     \cline{3-27}
    &&\bf 8    &65&	33&	17&	9&	33&	17&	9&	65&	33&	17&	9&	33&	17&	9&	71&	36&	17&	9&	38&	19&	10&	36&	18&	11  \\    
     \cline{3-27}
    &&\bf 16   & 49&	25&	13&	7&	25&	13&	7&	49&	25&	13&	7&	25&	13&	7&	60&	30&	13&	7&	32&	16&	8&	30&	15&	9 \\
      \cline{3-27}
   & &\bf 32   & 36&	18&	9&	5&	18&	9&	5&	36&	18&	9&	5&	18&	9&	5&	45&	24&	11&	5&	27&	13&	6&	24&	12&	7  \\
     \cline{3-27}
     
    &&\bf 64    &27&	14&	7&	4&	14&	7&	4&	27&	14&	7&	4&	14&	7&	4&	34&	18&	8&	4&	22&	10&	5&	18&	9&	6 \\
   \hline
  \hline
       \multirow{6}{*}{\rotatebox[origin=c]{90}{\sky}}
   &\multirow{6}{*}{305} &\bf 2   &237&	119&	60&	30&	119&	60&	31&	233&	112&	56&	28&	112&	56&	28&	233&	143&	75&	34&	160&	81&	41&	143&	75&	35 \\
    \cline{3-27}
   & &\bf 4 &135&	68&	34&	17&	68&	34&	18&	131&	66&	33&	17&	66&	33&	17&	193&	124&	49&	20&	135&	68&	36&	124&	47&	22 \\
     \cline{3-27}
    &&\bf 8    &83&	42&	21&	11&	42&	21&	11&	83&	42&	21&	11&	42&	21&	11&	127&	84&	31&	12&	106&	54&	27&	84&	32&	14 \\    
     \cline{3-27}
    &&\bf 16   &54&	27&	14&	7&	27&	14&	7&	54&	27&	14&	7&	27&	14&	7&	94&	56&	20&	8&	80&	41&	20&	56&	20&	10 \\
      \cline{3-27}
   & &\bf 32   & 39&	20&	10&	5&	20&	10&	5&	39&	20&	10&	5&	20&	10&	5&	62&	37&	13&	6&	58&	28&	12&	37&	15&	7 \\
     \cline{3-27}
     
    &&\bf 64    &29&	15&	8&	4&	15&	8&	4&	29&	15&	8&	4&	15&	8&	4&	43&	25&	9&	4&	41&	21&	8&	25&	11&	6 \\
   \hline
  \hline
       \multirow{6}{*}{\rotatebox[origin=c]{90}{\skyt}}

   &\multirow{6}{*}{245} &\bf 2   &216&	108&	56&	28&	108&	56&	29&	201&	103&	52&	26&	103&	52&	25&	203&	116&50&	26&	117&	69&	30&	116&	60&	34 \\
    \cline{3-27}
   & &\bf 4 &170&	85&	43&	22&	84&	43&	x&	149&	77&	39&	20&	77&	39&	20&	170&	134&	53&	18&	123&	68&	23&	130&	58&	29 \\
     \cline{3-27}
    &&\bf 8    &108&	54&	27&	14&	55&	28&	x&	101&	51&	26&	13&	51&	26&	13&	140&	99&	32&	14&	124&	55&	16&	99&	36&	18 \\    
     \cline{3-27}
    &&\bf 16   &61&	31&	15&	8&	30&	16&	9&	58&	29&	15&	8&	29&	15&	8&	105&	63&	19&	8&	101&	37&	11&	63&	21&	10 \\
      \cline{3-27}
   & &\bf 32   & 34&	17&	9&	5&	18&	9&	5&	34&	17&	9&	5&	17&	9&	5&	77&	38&	11&	5&	71&	23&	7&	38&	13&	7 \\
     \cline{3-27}    
    &&\bf 64    &23&	12&	6&	3&	12&	6&	3&	23&	12&	6&	3&	12&	6&	3&	53&	24&	7&	4&	48&	16&	5&	24&	9&	5    \\
   \hline
  \hline  
       \multirow{6}{*}{\rotatebox[origin=c]{90}{\ani}}
   &\multirow{6}{*}{73} &\bf 2   &70&	35&	18&	9&	35&	18&	9&	70&	35&	18&	9&	35&	18&	9&	70&	35&	18&	9&	35&	18&	9&	35&	18&	9 \\
    \cline{3-27}
   & &\bf 4 &63&	32&	16&	8&	32&	16&	8&	63&	32&	16&	8&	32&	16&	8&	66&	33&	16&	9&	33&	17&	9&	33&	17&	10 \\
     \cline{3-27}
    &&\bf 8    &57&	29&	15&	8&	29&	15&	8&	57&	29&	15&	8&	29&	15&	8&	59&	30&	15&	8&	30&	15&	8&	30&	17&	11\\    
     \cline{3-27}
    &&\bf 16   &50&	25&	13&	7&	25&	13&	7&	50&	25&	13&	7&	25&	13&	7&	54&	27&	14&	7&	28&	14&	7&	27&	16&	10 \\
      \cline{3-27}
   & &\bf 32   &43&	22&	11&	6&	22&	11&	6&	43&	22&	11&	6&	22&	11&	6&	51&	25&	12&	6&	25&	13&	7&	25&	15&	9\\
     \cline{3-27}    
    &&\bf 64    &35&	18&	9&	5&	18&	9&	5&	35&	18&	9&	5&	18&	9&	5&	44&	21&	10&	5&	23&	11&	6&	21&	13&	7\\
   \hline
  \hline
  \end{tabular}
\end{table}  

\begin{table}[H]
\setlength{\tabcolsep}{2pt}
\caption{\label{tab:precECG} Comparison of the convergence of different Block Jacobi with Complete Cholesky preconditioned E-CG versions  (s-step SRE-CG, CA SRE-CG with Algorithm \ref{alg:CA-Arnoldi2}, s-step SRE-CG2, CA SRE-CG2 with \ref{alg:CA-Arnoldi} or \ref{alg:CA-Arnoldi2}, s-step MSDO-CG, CA MSDO-CG with Algorithm \ref{alg:CA-Arnoldi} and Algorithm \ref{alg:CA-Arnoldi2}) with respect to number of partitions $t$ and $s$ values.}
  \centering
  %\tabsize    
    \renewcommand{\arraystretch}{1.2}
  \begin{tabular}{||c||c||c||c|c|c|c||c|c|c|||c|c|c|c||c|c|c|||c|c|c|c||c|c|c||c|c|c||}
    \cline{4-27} 
     \cline{4-27}
     \multicolumn{1}{c}{}& \multicolumn{1}{c}{}& \multicolumn{1}{c||}{} & \multicolumn{7}{c|||}{\multirow{1}{*}{\bf SRE-CG }} & \multicolumn{7}{c|||}{\multirow{1}{*}{\bf SRE-CG2}}&  \multicolumn{10}{c||}{\multirow{1}{*}{\bf  MSDO-CG}}\\     \cline{4-27} 
    \multicolumn{1}{c}{}& \multicolumn{1}{c}{}& \multicolumn{1}{c||}{} & \multicolumn{4}{c||}{\multirow{2}{*}{\bf s-step  }} & \multicolumn{3}{c|||}{\multirow{2}{*}{\bf CA Alg7}}& \multicolumn{4}{c||}{\multirow{2}{*}{\bf s-step }}& \multicolumn{3}{c|||}{\multirow{2}{*}{\bf CA Alg5/7}}& \multicolumn{4}{c||}{\multirow{2}{*}{\bf s-step }}& \multicolumn{3}{c||}{\multirow{2}{*}{\bf CA Alg5}}& \multicolumn{3}{c||}{\multirow{2}{*}{\bf CA Alg7}}\\
        \multicolumn{1}{c}{}& \multicolumn{1}{c}{}& \multicolumn{1}{c||}{}& \multicolumn{4}{c||}{\multirow{2}{*}{}} & \multicolumn{3}{c|||}{\multirow{2}{*}{}} & \multicolumn{4}{c||}{\multirow{2}{*}{}} & \multicolumn{3}{c|||}{\multirow{2}{*}{}} & \multicolumn{4}{c||}{\multirow{2}{*}{}} & \multicolumn{3}{c||}{\multirow{2}{*}{}}& \multicolumn{3}{c||}{\multirow{2}{*}{}}\\
    \cline{2-27}
    \multicolumn{1}{c||}{} & \multicolumn{1}{c||}{\bf PCG} & \backslashbox{$\bf t$}{$\bf s$} & \bf 1 &\bf  2  &\bf   4 &\bf  8  &\bf  2  &\bf   4 &\bf  8 &\bf  1 &\bf  2  &\bf  4  &\bf  8 &\bf  2 &\bf  4 &\bf 8  &\bf  1 &\bf  2  &\bf   4 &\bf  8 &\bf  2  &\bf   4 &\bf  8 &\bf  2  &\bf   4 &\bf  8 \\
    \hline \hline    
             \multirow{6}{*}{\rotatebox[origin=c]{90}{\bf \poiss}}
   &\multirow{6}{*}{67} &\bf 2   &60&	30&	15&	8&	30&	15&	8&	60&	30&	15&	8&	30&	15&	8&	61&	31&	15&	8&	32&	16&	8& 31&	16&	9  \\
    \cline{3-27}
   & & \bf 4    &51&	26&	13&	7&	26&	13&	7&	51&	26&	13&	7&	26&	13&	7&	53&	27&	14&	7& 27&	14&	7&	27&	14&	8 \\
     \cline{3-27}
    &&\bf  8    &42&	21&	11&	6&	21&	11&	6&	42&	21&	11&	6&	21&	11&	6&	45&	23&	11&	6&24&	12&	6&	23&	12&	7   \\    
     \cline{3-27}
    &&\bf 16   &33&	17&	9&	5&	17&	9&	5&	33&	17&	9&	5&	17&	9&	5&	37&	20&	9&	5&21&	10&	5&	20&	10&	6 \\
      \cline{3-27}
   & &\bf 32   &25&	13&	7&	4&	13&	7&	4&	25&	13&	7&	4&	13&	7&	4&	30&	16&	7&	4& 17&	9&	4&	16&	8&	5 \\
     \cline{3-27}
    &&\bf 64    &20&	10&	5&	3&	10&	5&	3&	20&	10&	5&	3&	10&	5&	3&	23&	12&	6&	3&	14&	7&	3& 12&	6&	4      \\
   \hline
  \hline
         \multirow{6}{*}{\rotatebox[origin=c]{90}{\nh}}
   &\multirow{6}{*}{92} &\bf 2   & 74&	37&	19&	10&	37&	19&	10&	74&	37&	19&	10&	37&	19&	10&	81&	40&	19&	10&41&	21&	11&	40&	21&	12 \\
    \cline{3-27}
   & &\bf 4 & 61&	31&	16&	8&	31&	16&	8&	61&	31&	16&	8&	31&	16&	8&	66&	33&	17&	8&	33&	17&	9& 33&	17&	10\\
     \cline{3-27}
    &&\bf 8    &51&	26&	13&	7&	26&	13&	7&	51&	26&	13&	7&	26&	13&	7&	55&	28&	13&	7&29&	14&	7&	28&	14&	9   \\    
     \cline{3-27}
    &&\bf 16   & 39&	20&	10&	5&	20&	10&	5&	39&	20&	10&	5&	20&	10&	5&	44&	23&	11&	5&	24&	12&	6&23&	12&	8 \\
      \cline{3-27}
   & &\bf 32   & 30&	15&	8&	4&	15&	8&	4&	30&	15&	8&	4&	15&	8&	4&	34&	19&	8&	4&20&	10&	5&	19&	9&	6  \\
     \cline{3-27}
     
    &&\bf 64    &23&	12&	6&	3&	12&	6&	3&	23&	12&	6&	3&	12&	6&	3&	26&	14&	6&	3&16&	8&	4&	14&	8&	5 \\
   \hline
  \hline
       \multirow{6}{*}{\rotatebox[origin=c]{90}{\sky}}
   &\multirow{6}{*}{264} &\bf 2   &193&	97&	48&	24&	97&	48&	27&	183&	92&	46&	23&	92&	46&	23&	189&	121&	56&	25&	118&	60&	33& 121&	58&	27 \\
    \cline{3-27}
   & &\bf 4 &105&	53&	27&	14&	53&	27&	14&	105&	53&	27&	14&	53&	27&	14&	146&	91&	37&	16&	106&	54&	27& 91&	38&	17\\
     \cline{3-27}
    &&\bf 8    &66&	33&	17&	9&	33&	17&	9&	66&	33&	17&	9&	33&	17&	9&	98&	64&	23&	10&	80&	39&	20& 64&	23&	12\\    
     \cline{3-27}
    &&\bf 16   &44&	22&	11&	6&	22&	11&	6&	44&	22&	11&	6&	22&	11&	6&	70&	41&	15&	7&58&	28&	13&	41&	17&	8\\
      \cline{3-27}
   & &\bf 32   & 31&	16&	8&	4&	16&	8&	4&	31&	16&	8&	4&	16&	8&	4&	48&	26&	10&	5&37&	18&	7&	26&	11&	6\\
     \cline{3-27}
     
    &&\bf 64    &19&	10&	5&	3&	10&	5&	3&	19&	10&	5&	3&	10&	5&	3&	30&	17&	6&	3&21&	10&	4&	17&	7&	4 \\
   \hline
  \hline
       \multirow{6}{*}{\rotatebox[origin=c]{90}{\skyt}}
   &\multirow{6}{*}{225} &\bf 2   & 181&	91&	48&	24&	94&	48&	26&	173&	87&	46&	23&	87&	46&	23&	186&	106&	48&	25&114&	61&	28&	106&	49&	32\\
    \cline{3-27}
   & &\bf 4 & 139&	70&	37&	19&	72&	38&	x&	130&	65&	34&	17&	65&	34&	18&	154&	113&	43&	19&	112&	60&	22& 113&	47&	24\\
     \cline{3-27}
    &&\bf 8    &80&	40&	20&	10&	40&	20&	14&	77&	39&	20&	10&	39&	20&	10&	117&	76&	26&	11&99&	48&	15&	76&	28&	14 \\    
     \cline{3-27}
    &&\bf 16   &45&	23&	12&	6&	23&	12&	6&	45&	23&	12&	6&	23&	12&	6&	91&	52&	15&	6&87&	32&	10&	52&	17&	8\\
      \cline{3-27}
   & &\bf 32   & 29&	15&	8&	4&	15&	8&	4&	29&	15&	8&	4&	15&	8&	4&	62&	29&	9&	4&57&	20&	6&	29&	10&	6\\
     \cline{3-27}    
    &&\bf 64    &20&	10&	5&	3&	10&	5&	3&	20&	10&	5&	3&	10&	5&	3&	44&	18&	7&	3&38&	13&	4&	18&	7&	4    \\
   \hline
  \hline  
       \multirow{6}{*}{\rotatebox[origin=c]{90}{\ani}}
   &\multirow{6}{*}{69} &\bf 2   & 66&	33&	17&	9&	33&	17&	9&	66&	33&	17&	9&	33&	17&	9&	66&	33&	17&	9& 33&	17&	9&	33&	17&	9 \\
    \cline{3-27}
   & &\bf 4 &61&	31&	16&	8&	31&	16&	8&	61&	31&	16&	8&	31&	16&	8&	61&	31&	15&	8&31&	16&	8&	31&	16&	10\\
     \cline{3-27}
    &&\bf 8    &56&	28&	14&	7&	28&	14&	7&	56&	28&	14&	7&	28&	14&	7&	58&	29&	15&	8&29&	15&	8&	29&	16&	11
\\    
     \cline{3-27}
    &&\bf 16   &49&	25&	13&	7&	25&	13&	7&	49&	25&	13&	7&	25&	13&	7&	54&	27&	13&	7&	28&	14&	7&27&	16&	10 \\
      \cline{3-27}
   & &\bf 32   &42&	21&	11&	6&	21&	11&	6&	42&	21&	11&	6&	21&	11&	6&	50&	24&	12&	6&	25&	13&	6&24&	14&	10\\
     \cline{3-27}    
    &&\bf 64    &35&	18&	9&	5&	18&	9&	5&	35&	18&	9&	5&	18&	9&	5&	44&	20&	10&	5&	21&	11&	5&20&	12&	7\\
   \hline
  \hline
  \end{tabular}
\end{table}

 \section{Parallelization and Expected performance} \label{sec:par}
In this section, we briefly describe the parallelization of the unpreconditioned and preconditioned, s-step and CA SRE-CG, SRE-CG2, and MSDO-CG methods, assuming that the algorithms are executed on a distributed memory machine with $t$ processors. Then, we compare the performance of the s-step and CA methods with respect to the SRE-CG, SRE-CG2, and MSDO-CG methods. Finally, we compare the expected performance of the CA enlarged CG versions with respect to the classical CG, in terms of memory, flops and communication.

In what follows, we assume that the estimated runtime of an algorithm with a total of  $z$  computed flops  and $s$ sent messages, each of size $k$, is $\gamma_c z+ \alpha_c s + \beta_c sk$, where $\gamma_c $ is the inverse floating-point rate (seconds per floating-point operation), $\alpha_c$ is the latency (seconds), and $\beta_c$ is the inverse bandwidth (seconds per word). Moreover, unless specified otherwise, we assume that the number of processors is equal to the number of partitions $t$.

 \subsection{Unpreconditioned Methods} \label{sec:unpm}

The unpreconditioned s-step SRE-CG and s-step SRE-CG2 parallelization is similar to that of SRE-CG and SRE-CG2 described in \cite{EKS}, with the difference that the s-step versions send $(s-1)log(t)$ less messages and words per s-step iteration. Moreover, the s-step MSDO-CG's algorithm is similar to that of s-step SRE-CG in structure. Thus the number of messages sent in parallel is the same as that of s-step SRE-CG2. We assume that SRE-CG, SRE-CG2, and MSDO-CG converge in $k$ iterations and the corresponding s-step versions converge in $k_s = \ceil{\frac{k}{s}}$ iteration. Thus, $ 5s{k_s}log(t) + k_slog(t)\approx 5{k}log(t) + \frac{k}{s}log(t)$ messages are sent in parallel in the s-step versions, compared to $6{k}log(t)$ messages. This leads to a $ \frac{(s- 1)100}{6s} \%$ reduction in communication, without increasing the number of computed flops.  For example, for $s=3$,  $11.11\%$ reduction is achieved in the s-step versions, and $15\%$ reduction for $s=10$.

The difference between the parallelization of unpreconditioned CA MSDO-CG and unpreconditioned CA SRE-CG2 is in the basis construction. In CA MSDO-CG each of the $t$ processors can compute the $s$ basis vectors
$$T_i(r_{k-1}), AT_i(r_{k-1}), A^2T_i(r_{k-1}),..., A^{s-1}T_i(r_{k-1})$$
 independently from other processors, where $T_i(r_{k-1})$ is a vector of all zeros except at $n/t$ entries that correspond to the $i^{th}$ domain of the matrix $A$. Thus, there is no need for communication avoiding kernels. To compute the $s$ vectors without any communication, processor $i$ needs $T_i(r_{k-1})$, the row-wise part of the vector $r_{k-1}$ corresponding to the $i^{th}$ domain $D_i$, and a part of the matrix $A$ depending on $s$ and the sparsity pattern of $A$. Specifically, processor $i$ needs the column-wise part of $A$ corresponding to $R(G(A), D_i ,s)$, the set of vertices in the graph of $A$ reachable
by paths of length at most $s$ from any vertex in $D_i$. 
 
 On the other hand, in CA SRE-CG2 at iteration $k$, the $st$ basis vectors $$AW_{(k-1)s}, A^2W_{(k-1)s},..,A^sW_{(k-1)s}$$  have to be computed, where $W_{(k-1)s}$ is a block of $t$ dense vectors. Similarly to CA MSDO-CG, each of the $t$ processors can compute the $s$ basis vectors
$$AW_{(k-1)s}(:,i), A^2W_{(k-1)s}(:,i),..,A^sW_{(k-1)s}(:,i)$$
 independently from other processors. But processor $i$ needs the full matrix $A$ and the vector $W_{(k-1)s}(:,i)$. Another alternative is to use a block version of the matrix powers kernel, where processor $i$ computes a row-wise part of the $s$ blocks without any communication, by performing some redundant computations. Moreover, as discussed in section \ref{sec:SRNum}, for numerical stability purposes, $AW_{(k-1)s}$ has to be A-orthonormalized before proceeding in the basis construction. This increases the number of messages sent. 
 
  Then, all of the computed $st$ vectors in  CA MSDO-CG and  CA SRE-CG2,  are A-orthonormalized against the previous $st(k-1)$ vectors using CGS2 (Algorithm 18 in \cite{sophiethesis}), and against themselves using A-CholQR (Algorithm
21 in \cite{sophiethesis}) or Pre-CholQR (Algorithm 23 in \cite{sophiethesis}). The parallelization of the A-orthonormalization algorithms is described in details for a block of $t$ vectors in \cite{sophiethesis}. For a block of $st$ vectors, the same number of messages is sent but with more words. Thus, $ 5{k_s}log(t) + k_slog(t)\approx  6\frac{k}{s}log(t)$ messages are sent in parallel in CA MSDO-CG, leading to a $\frac{(s-1)100}{s} \%$ reduction in communication  as compared to MSDO-CG (for $s=3 \implies 66.6\%$ reduction). Whereas, is CA SRE-CG $ 2*5{k_s}log(t) + k_slog(t)\approx  11\frac{k}{s}log(t)$ messages are sent, leading to a $\frac{(6s-11)100}{6s} \%$ reduction in communication (for $s=3 \implies 17.4\%$ reduction).  

As for the unpreconditioned CA SRE-CG, its parallelization is exactly the same as that of CA SRE-CG2, where the same number of messages is sent in parallel. However, less words are sent per message, since in CA SRE-CG the $st$ computed vectors are A-orthonormalized against the previous $(s+1)t$ vectors.

Thus the s-step and CA versions of the enlarged CG methods reduce communication as compared to their corresponding enlarged versions for the same number of processors and the same number of partitions $t$. All the s-step versions are comparable in terms of numerical stability and communication reduction. However, CA MSDO-CG is a better choice since it reduces communication the most. But the question that poses itself is: ``Is it better, in terms of communication, to double $t$ or merge two iterations of MSDO-CG?". 

The number of flops performed per iteration in the s-step and CA MSDO-CG for $s = 2^i$ and a given $t$, is comparable to that of the  MSDO-CG algorithms where we have $2^i t$ partitions. This is due the fact that in both versions, we are constructing and A-orthonormalizing $2^i t$ basis vectors per iteration, for $\overline{\mathscr{K}}_{k,t,2^i}$ (s-step and CA MSDO-CG versions) and $\overline{\mathscr{K}}_{k,2^it,1}$ (MSDO-CG). On the other hand, based on the observed results in sections \ref{sec:SRNum} and \ref{sec:precconv}, by doubling $t$ in any of the enlarged CG methods, the number of iterations needed for convergence is not halved, but on average it is $25\%$ less. Whereas, in s-step MSDO-CG and CA MSDO-CG by doubling $s$ the number of iterations is halved (up to some value $s$).

The number of processors in the MSDO-CG, s-step MSDO-CG, and CA MSDO-CG could be equal, a multiple or a divisor of the number of partitions $t$. In the first case, we assume that the number of processors is equal to $t$. Let $k$ be the number of iterations needed for convergence of MSDO-CG, where $t$ basis vectors are computed per iteration. Then, the number of messages sent in parallel in MSDO-CG where we have $2^i t$ partitions, is $6(0.75)^i k log(t)$. 
Whereas, for $s=2^i$, $[5 + (0.5)^i]klog(t)$ messages are sent in parallel in s-step MSDO-CG, and $6(0.5)^i k log(t)$ messages are sent in CA MSDO-CG. But, for $i \geq 1$, we have that
$$6(0.5)^i k log(t) < 6(0.75)^i k log(t) < [5 + (0.5)^i]klog(t). $$
 Thus, in this case it is clear that doubling $s$ and using the CA version is better than doubling the number of partitions in MSDO-CG, which is better than using the s-step version.

In the second case, we assume that the number of processors is equal to the number of partitions. Then, the number of messages sent in parallel in MSDO-CG where we have $2^i t$ processors and partitions, is $6(0.75)^i k log(2^i t)$. However, in s-step MSDO-CG, and CA MSDO-CG we assume that we have $t$ processors and $s=2^i$. 
In this case, s-step MSDO-CG sends less messages than MSDO-CG, if and only if,
\begin{eqnarray}
6(0.75)^i(i+log(t)) &>& [5 + (0.5)^i]log(t) \nonumber\\
\iff 6i(0.75)^i &>& [5-6(0.75)^i + (0.5)^i] log(t).  \label{2i}
\end{eqnarray}
The inequality \eqref{2i} is valid for $i=1$ and $t=2,4,8,16$. This means that for $s=2$ s-step MSDO-CG requires less communication with $t=2,4,8,16$ processors/partitions than MSDO-CG with $2t$ processors/partitions. Hence, assuming that communication is much more expensive than flops, it is better to merge 2 iterations of MSDO-CG and compute a basis for $\overline{\mathscr{K}}_{k,t,2}$, than double the $t$ value and compute  a basis for $\overline{\mathscr{K}}_{k,2t,1}$. Moreover, \eqref{2i} is valid for $i=2$ ($s=4$) and $t=2,4,8$, and for $i=3,4$ ($s=8,16$) and $t=2,4$. %  for $i=5,6,7$ and $t=2$
On the other hand, CA MSDO-CG sends less messages than MSDO-CG for all values of $i \geq 1$ and $t \geq 2$, since 
\begin{eqnarray}
6(0.75)^i(i+log(t)) &>& 6(0.5)^i log(t) \nonumber\\
\iff i(0.75)^i &>& ((0.5)^i -(0.75)^i) log(t).  \label{2i2}
\end{eqnarray}

 \subsection{Preconditioned Methods} \label{sec:pm}
The only difference between the preconditioned and unpreconditioned algorithms is in the matrix - block of vectors multiplication where the preconditioner is applied. Thus, if the preconditioned matrix - block of vectors multiplication can be performed without communication, then the same number of messages will be sent per iteration.
In the s-step versions, the preconditioner is applied twice, $L^{-t}[{T}(\widehat{r}_{k-1})]$ and $W_{i+1} = M^{-1}AW_i$, where $\widehat{r}_{k-1} = L^{-1}r_{k-1}$, $M = LL^t$ and $W_i$ is a dense $n\times t$ matrix. The parallelization of these ``multiplications" depends on the type of the preconditioner and the sparsity pattern of $A$. 

For example, let $L$ be a block diagonal lower triangular matrix with $t$ blocks, $L_i$, for $i=1,..,t$. Then, ${T}_i(\widehat{r}_{k-1}) = {T}_i(L^{-1}{r}_{k-1}) = L^{-1}{T}_i({r}_{k-1})$ is an all zero vector except the entries corresponding to domain $D_i$, which are obtained by solving $z_i = L_i^{-1}{r}_{k-1}(D_i)$ using forward substitution. Thus, processor $i$ needs the $i^{th}$ diagonal block $L_i$ and $T_i(r_{k-1})$ to compute ${T}_i(\widehat{r}_{k-1})$.
Similarly, computing $L^{-t}{T}_i(\widehat{r}_{k-1})$ is reduced to computing $L_i^{-t}z_i$ using backward substitution. Thus, processor $i$ computes the vector $L^{-t}[{T}_i(\widehat{r}_{k-1})]$ without any communication.  As for $W_{j+1} = M^{-1}AW_j$, processor $i$ computes the row-wise part of $W_{j+1}$ corresponding to domain $D_i$, using the row-wise part of $A$, $L_i$, $L_i^t$, and the row-wise part of $W_j$ corresponding to $\delta_i = R(G(A),D_i,1)$. Thus, processor $i$ computes $Z_i = A(D_i,\delta_i)W_j(\delta_i,:)$ without any communication, and solves for $W_{j+1}(D_i,:) = (L_iL_i^t)^{-1}Z_i$ using a backward and forward  substitution.

In preconditioned CA MSDO-CG,  at the $k^{th}$ iteration the $st$ vectors 
$$L^{-t}T(\hat{r}_{k-1}), M^{-1}AL^{-t}T(\hat{r}_{k-1}), (M^{-1}A)^2L^{-t}T(\hat{r}_{k-1}),..., (M^{-1}A)^{s-1}L^{-t}T(\hat{r}_{k-1})$$
are computed. 
Assuming that $L$ is a block diagonal lower triangular matrix, then each of the $t$ processors can compute the $s$ basis vectors
$$L^{-t}T_i(\hat{r}_{k-1}), M^{-1}AL^{-t}T_i(\hat{r}_{k-1}), (M^{-1}A)^2L^{-t}T_i(\hat{r}_{k-1}),..., (M^{-1}A)^{s-1}L^{-t}T_i(\hat{r}_{k-1})$$
 independently from other processors, but using a relatively big column-wise part of $A$ and $M$, depending on $s$ and the sparsity patterns of $A$ and $L$.  
 Another alternative is that each of the $t$ processors computes the row-wise part of the $st$ vectors corresponding to $D_i$ without communication using a preconditioned block version of the matrix powers kernel. However, the same ``relatively big" column-wise part column-wise part of $A$ and $M$ is needed.  
 
 To reduce the memory storage needed per processor, one option is to overlap communication with computation in the preconditioner's application. Let $W_1 = L^{-t}T(\hat{r}_{k-1})$, and $W_{j+1} = M^{-1}A W_{j}$ for $j\geq 1$. Each processor $i$ can compute $W_1(D_i,:)$ independently, since $W_1(D_i,:)$ is all zeros except the $i^{th}$ column which is equivalent to solving for $L_i^{-t}T_i(\hat{r}_{k-1})$. To compute $W_{j+1}(D_i,:) = L_i^{-t}L_i^{-1} Z_i$, where $Z_i = A(D_i,\delta_i)W_j(\delta_i,:)$, processor $i$ needs part of $W_j(\delta_i,:)$ from neighboring processors. This local communication occurs once $W_j$ is computed, and it is overlapped with the computation of $A(D_i,D_i)W_j(D_i,:)$. Then the remaining part of the multiplication is performed once the messages are received from neighboring processors. In this case, processor $i$ only needs $A(D_i,\delta_i)$, and $L_i$. And there are $s-1$ communication phases, once before the last $s-1$ preconditioned matrix multiplications. Even though these local communications are hidden with computations, but they might require some additional time. However, the gain in communication reduction from replacing $s$ A-orthonormalization procedures by just one, overweighs this ``possible" additional communication, as the A-orthonormalization requires global communication.
 
In preconditioned CA SRE-CG and CA SRE-CG2, at the first iteration $$L^{-t}T(\hat{r}_{0}), M^{-1}AL^{-t}T(\hat{r}_{0}), (M^{-1}A)^2L^{-t}T(\hat{r}_{0}),..., (M^{-1}A)^{s-1}L^{-t}T(\hat{r}_{0})$$ are computed, but at the $k^{th}$ iteration the $st$ vectors $$M^{-1}AW, (M^{-1}A)^2W,..., (M^{-1}A)^{s}W$$ are computed, where $W = W_{(k-1)s}$ is a dense $n \times t$ matrix. The communication pattern and parallelization of the preconditioned matrix multiplication is the same as that of CA MSDO-CG, with the exception that for $k>1$ an additional local communication is required for $M^{-1}AW$.  Moreover, the communication reduction in the preconditioned CA SRE-CG2 is comparable to that of CA MSDO-CG, since once preconditioned, SRE-CG2 with Algorithm \ref{alg:CA-Arnoldi} converges and scales even for $s=8$, as discussed in section \ref{sec:precconv}.

\subsection{Expected Performance}\label{expperf}
By merging $s$ iterations of the enlarged CG versions, communication is reduced in the corresponding s-step and CA versions as discussed in sections \ref{sec:unpm} and \ref{sec:pm}. Moreover, the enlarged CG versions converge faster in terms of iterations than CG by enlarging the Krylov Subspace. However, are these reductions enough to obtain a method that converges faster than CG in terms of runtime, using comparable resources?

Conjugate Gradient is known for its short recurrence formulae and the limited memory storage.
In preconditioned CG (Algorithm \ref{alg:pCG}), if  processor $i$ computes part of the vectors $p_k(D_i), w(D_i), x_k(D_i), r_k(D_i), \widehat{r}_k(D_i)$, then it needs $A(D_i,:)$, $L_i$ and $b(D_i)$, assuming that $M = LL^t$ is a block diagonal matrix. Moreover, two global communications are needed per iteration to perform the dot products  $p^tw, {\rho}_k, \widehat{\rho}_k$, and local communication with neighboring processors is needed to compute $w(D_i) = A(D_i,\delta_i)p(\delta_i)$, where $\delta_i = R(G(A),D_i,1)$. Given that there is a total of $m$ processors, $2log(m)$ messages are sent per CG iteration without considering local communication.
  
 \begin{algorithm}[h!]
\centering
\caption{Preconditioned CG}
{\renewcommand{\arraystretch}{1.3}
\begin{algorithmic}[1]
\Statex{\textbf{Input:} $A$, $M = LL^t$, $b$, $x_0$, $\epsilon$, $k_{max}$ }
\Statex{\textbf{Output:} $x_k$, the approximate solution of the system $L^{-t}AL^t(L^{-t}x)=L^{-t}b$}
\State$r_0 = b - Ax_0$, $\rho_0 = ||r_0||^2_2$, $\widehat{r}_0 = M^{-1}r_{0}$, $\widehat{\rho}_0 = r_0^t \widehat{r}_0$, $k = 1$; 
%\State 
\While {( $\sqrt{\rho_{k-1}} > \epsilon \sqrt{\rho_{0}}$ and $k \leq k_{max}$ )}
\If {($k==1$)} $p =  \widehat{r}_0$
\Else  $\;\;\; \beta = \frac{\widehat{\rho}_{k-1}}{\widehat{\rho}_{k-2}}$,  $p =  \widehat{r}_k + \beta p$
\EndIf 
\State $w = Ap$,  \;\; $\alpha =  \dfrac{\widehat{\rho}_{k-1}}{p^tw}$
\State $x_k = x_{k-1} + \alpha p $, \;\; $r_k = r_{k-1} - \alpha w $, \;\; $\widehat{r}_k = M^{-1}r_{k}$
\State $\rho_k = ||r_{k}||^2_2$, \;\; $\widehat{\rho}_k = r_k^t \widehat{r}_k$, \;\; $k = k+1$  
\EndWhile
\end{algorithmic}}
\label{alg:pCG}
\end{algorithm}

Similarly to CG, the SRE-CG, SRE-CG2, and MSDO-CG versions have short recurrence formulae. 
However, in terms of memory, the SRE-CG versions are the best choice since a limited number of vectors, depending only on $t$ and $s$, need to be stored. In SRE-CG, s-step SRE-CG, and CA SRE-CG, $(3t)$ vectors, $(s+2)t$ vectors, and $(2s+1)t$ vectors are stored respectively. Whereas, in SRE-CG2 and MSDO-CG version, $stk$ vectors need to be stored, where $k$ is the number of iterations needed for convergence which is not known a priori.

Given that $s = 2^i$, the number of partitions is $t = 2^j$, and a total of $m$ processors run the algorithms where $m$ is a multiple, divisor or equal to $t = 2^j$ for $j,i \geq 1$;  then, $2klog(m)$ messages are sent in CG with a total of $k$ iterations till convergence.  As for SRE-CG, a total of $6(0.75)^{j}klog(m)$ messages are sent, assuming that as $t$ is doubled the number of iterations is reduced by $25\%$ on average. Whereas in s-step and CA SRE-CG, a total of $[5+(0.5)^i](0.75)^{j}klog(m)$ and $11(0.5)^i(0.75)^{j}klog(m)$ messages are sent respectively.

As compared to CG, SRE-CG, s-step SRE-CG, and CA SRE-CG communicate less in total when 
\begin{eqnarray}
2klog(m) &>& 6(0.75)^{j}klog(m) \label{sre}\\
2klog(m) &>& [5+(0.5)^i](0.75)^{j}klog(m)\label{sstepsre}\\
2klog(m) &>& 11(0.5)^i(0.75)^{j}klog(m) \label{casre}
\end{eqnarray}
respectively. For $j\geq 4$, inequality \eqref{sre} is satisfied, i.e. SRE-CG reduces communication with respect to CG for number of partitions $t \geq 16$. Similarly, s-step SRE-CG further reduces communication with respect to CG for $s=2,4,8$ and $t \geq 16$. Whereas CA SRE-CG reduces communication for $s=2$ and $j\geq 4$ ($t \geq 16$), $s = 4$ and $j\geq 2$ ( $t \geq 4$), and $s = 8$ and $j\geq 1$ ( $t \geq 2$).

Hence, for $s = 2^i$ and $t=2^j$ in the  SRE-CG, s-step SRE-CG, and CA SRE-CG, the reduction in communication with respect to CG is respectively ($j = 5$ and $i = 2$)
\begin{eqnarray} 
100-3(0.75)^j100\% &\qquad = \qquad&  (28.8\%), \nonumber \\
100-(2.5+0.5^{i+1})0.75^j100\%  &\qquad = \qquad& (37.7\%),\nonumber \\
100-5.5(0.5)^i(0.75)^{j}100\% & \qquad = \qquad & (67.37 \%). \nonumber 
\end{eqnarray} 

Thus, it is expected that SRE-CG, s-step SRE-CG, and CA SRE-CG will converge faster than CG in parallel considering the $s$ and $t$ values discussed above and that communication is much more expensive than flops. Moreover, even the s-step and CA versions of the SRE-CG2 and MSDO-CG are expected to converge faster than CG, but require much more memory storage per processor.   
%=====================================================================
 \section{Conclusion} \label{sec:conc}
%=====================================================================
In this paper, we introduced the s-step and communication avoiding versions of  SRE-CG, and  SRE-CG2, which are based on the enlarged Krylov subspace. We have also introduced a modified MSDO-CG version that is equivalent to MSDO-CG theoretically and numerically, but based on a modified enlarged Krylov subspace which allows the s-step and CA formulations. The split preconditioned s-step and CA versions are also presented in section \ref{sec:precCG}. 

The s-step and communication avoiding versions merge $s$ iterations of the enlarged CG methods into one iteration where denser operations are performed for less communication. Numerical stability of the s-step and CA version is tested in section \ref{sec:SRNum}, where as $s$ is doubled, the number of iterations needed for convergence in the s-step methods is roughly divided by two, even for $s\geq 10$. As for the CA methods, once the system is preconditioned, a similar scaling behavior is observed in section \ref{sec:precconv}. Accordingly, it is shown in sections \ref{sec:unpm} and \ref{sec:pm} that the s-step and CA versions reduce communication with respect to the corresponding enlarged methods for $s \geq 2$. 

Although the number of messages per iteration of the enlarged CG methods and their s-step and CA versions, is more than that of CG, however due to the reduction in the number of iterations in the enlarged versions, the total messages sent is less as discussed in section \ref{expperf}. This implies that all the enlarged CG variants should require less time to converge than CG. However, the SRE-CG variants are the most feasible candidates due to their limited memory storage requirements.

Future work will focus on implementing, testing, and comparing the runtime of the introduced enlarged CG versions on CPU's and GPU's with respect to existing similar methods.

\end{document}